\documentclass[%
    aip,
    sd,%
    amsmath,
    amssymb,
    preprint,%
    twocolumn,
    10pt,
    a4paper,
    nofootinbib,
    tightenlines
]{revtex4-1}

\usepackage
{
    graphicx,
    amssymb,
    amsmath,
    amsthm,
    dsfont, 
    xcolor,
    enumitem,
    tikz-cd, 
    mathtools,
    ifthen,
    mdframed, 
    xfrac, 
}

\usepackage[pdffitwindow=false,
            plainpages=false,
            pdfpagelabels=true,
            pdfpagemode=UseOutlines,
            pdfpagelayout=SinglePage,
            bookmarks=false,
            colorlinks=true,
            hyperfootnotes=false,
            linkcolor=blue,
            urlcolor=blue!30!black,
            citecolor=green!50!black]{hyperref}

\graphicspath{{pics/}}

\newtheorem{theorem}{Theorem}[section]

\newtheorem{lemma}[theorem]{Lemma}
\newtheorem{proposition}[theorem]{Proposition}
\newtheorem{definition}[theorem]{Definition}
\theoremstyle{definition}
\newtheorem{example}[theorem]{Example}
\newtheorem{remark}[theorem]{Remark}
\newtheorem{textalgorithm}[theorem]{Algorithm}

\newcommand{\ts}{\hspace*{0.1em}} 

\newcommand{\subfiguretitle}[1]{{\scriptsize{#1}} \\[0.25ex]}
\newcommand{\R}{\mathbb{R}}                                     
\newcommand{\innerprod}[2]{\left\langle #1,\, #2 \right\rangle} 
\newcommand{\dd}{\mathrm{d}}                                    
\newcommand{\pp}[1]{\mathbb{#1}}                                
\providecommand{\norm}[1]{\left\lVert #1 \right\rVert}          

\newcommand\xqed[1]{\leavevmode\unskip\penalty9999 \hbox{}\nobreak\hfill \quad\hbox{#1}}
\newcommand{\exampleSymbol}{\xqed{$\blacktriangle$}} 

\newcommand{\inspace}{\mathbb{X}} 
\newcommand{\outspace}{\mathbb{Y}} 

\newcommand{\rkhs}[1][]{\mathbb{H}_\mathit{\scriptscriptstyle #1}} 

\newcommand{\id}{I} 
\newcommand{\idop}{\mathcal{I}} 

\newcommand{\ebd}[1][]{
   \ifthenelse{\equal{#1}{}}{\mathcal{E}}{\mathcal{E}_{#1}}}

\newcommand{\pf}[1][]{
   \ifthenelse{\equal{#1}{}}{\mathcal{P}}{\mathcal{P}_{#1}}}
\newcommand{\epf}[1][]{
   \ifthenelse{\equal{#1}{}}{\widehat{\mathcal{P}}}{\widehat{\mathcal{P}}_{#1}}}
\newcommand{\ko}[1][]{
   \ifthenelse{\equal{#1}{}}{\mathcal{K}}{\mathcal{K}_{#1}}}
\newcommand{\eko}[1][]{
   \ifthenelse{\equal{#1}{}}{\widehat{\mathcal{K}}}{\widehat{\mathcal{K}}_{#1}}}

\newcommand{\cov}[1][]{\mathcal{C}_\mathit{\scriptscriptstyle #1}} 
\newcommand{\ecov}[1][]{\widehat{\mathcal{C}}_\mathit{\scriptscriptstyle #1}} 
\newcommand{\gram}[1][]{G_\mathit{\scriptscriptstyle #1}} 

\DeclareMathOperator{\mspan}{span}

\mathtoolsset{centercolon}

\definecolor{boxback}{gray}{0.95}

\allowdisplaybreaks



\begin{document}

\title{\large Kernel methods for detecting coherent structures in dynamical data}

\author{Stefan Klus}%
\email[]{stefan.klus@fu-berlin.de}
\affiliation{Department of Mathematics and Computer Science, \mbox{Freie Universit\"at Berlin, 14195 Berlin, Germany}}

\author{Brooke E. Husic}
\email[]{b.husic@fu-berlin.de}
\affiliation{Department of Mathematics and Computer Science, \mbox{Freie Universit\"at Berlin, 14195 Berlin, Germany}}
\affiliation{Department of Chemistry, \mbox{Stanford University, Stanford, CA, 94305, USA}}

\author{Mattes Mollenhauer}
\email[]{mattes.mollenhauer@fu-berlin.de}
\affiliation{Department of Mathematics and Computer Science, \mbox{Freie Universit\"at Berlin, 14195 Berlin, Germany}}

\author{Frank No\'e}
\email[]{frank.noe@fu-berlin.de}
\affiliation{Department of Mathematics and Computer Science, \mbox{Freie Universit\"at Berlin, 14195 Berlin, Germany}}

\begin{abstract}
We illustrate relationships between classical kernel-based dimensionality reduction techniques and eigendecompositions of empirical estimates of \emph{reproducing kernel Hilbert space} (RKHS) operators associated with dynamical systems. In particular, we show that kernel \emph{canonical correlation analysis} (CCA) can be interpreted in terms of kernel transfer operators and that it can be obtained by optimizing the \emph{variational approach for Markov processes} (VAMP) score. As a result, we show that coherent sets of particle trajectories can be computed by kernel CCA. We demonstrate the efficiency of this approach with several examples, namely the well-known Bickley jet, ocean drifter data, and a molecular dynamics problem with a time-dependent potential. Finally, we propose a straightforward generalization of \emph{dynamic mode decomposition} (DMD) called \emph{coherent mode decomposition} (CMD). Our results provide a generic machine learning approach to the computation of coherent sets with an objective score that can be used for cross-validation and the comparison of different methods.
\end{abstract}

\maketitle

\begin{quotation}
While coherent sets of particles are common in dynamical systems, they are notoriously challenging to identify. In this article, we leverage the combination of a suite of methods designed to approximate the eigenfunctions of transfer operators with kernel embeddings in order to design an algorithm for detecting coherent structures in Langrangian data. It turns out that the resulting method is a well-known technique to analyze relationships between multidimensional variables, namely kernel canonical correlation analysis. Our algorithm successfully identifies coherent structures in several diverse examples, including oceanic currents and a molecular dynamics problem with a moving potential. Furthermore, we show that a natural extension of our algorithm leads to a coherent mode decomposition, a counterpart to dynamic mode decomposition.
\end{quotation}

\section{Introduction}

Representing and learning effective, low-dimensional manifolds for complex, high-dimensional data is one of the cornerstones of machine learning and complex systems theory. In particular, for dynamical processes such as turbulent flows or molecular dynamics, it is known that much of their essential, long-time dynamics can be captured by linear models acting on low-dimensional manifolds, resulting in simpler, interpretable models, and potentially large savings in process simulation, prediction, and control. Due to their simplicity, linear methods to find low-dimensional subspaces are widely used, including \emph{principal component analysis} (PCA)~\cite{Hotelling33:PCA}, \emph{canonical correlation analysis} (CCA)~\cite{Hotelling_Biometrika36_CCA}, \emph{independent component analysis} (ICA)~\cite{Hyvarinen00:ICA}, \emph{time-lagged independent component analysis} (TICA) \cite{MS94, PPGDN13}, \emph{time-lagged canonical correlation analysis} (TCCA) \cite{WuNo17}, and \emph{dynamic mode decomposition} (DMD) \cite{SS08}.

Since the sought manifold is usually nonlinear in the direct state representation of the system, it is important to generalize the aforementioned methods to work in nonlinear feature spaces. Two particularly important classes of learning methods that go beyond applying linear algorithms to user-defined feature functions are neural network approaches\cite{MPWN:vampnets,LiEtAl_Chaos17_EDMD_DL,OttoRowley_LinearlyRecurrentAutoencoder} and kernel methods such as kernel PCA~\cite{Scholkopf98:KPCA}, kernel CCA~\cite{MRB01:CCA}, kernel ICA~\cite{Bach03:KICA}, kernel TICA~\cite{HZHM03:kTICA}, and kernel EDMD \cite{WRK15}. The basic idea of kernel methods is to represent data by elements in reproducing kernel Hilbert spaces associated with positive definite kernel functions.

The novel contribution of this work is to derive kernel methods for the identification of coherent structures in high-dimensional dynamical data. We do this by establishing deep mathematical connections between kernel methods that have
been proposed in machine learning and dynamical systems theory. Our main results are:
\begin{enumerate}[wide, itemindent=\parindent, itemsep=0ex, topsep=0.5ex]
\item We show that kernel CCA, when applied to dynamical data, admits a natural interpretation in terms of kernel transfer operators and that the resulting eigenvalue problems are directly linked to methods for the computation of coherent sets. Importantly, kernel CCA predates recent methods for coherent set identification.
\item We show kernel CCA is optimal in the variational approach for Markov processes
(VAMP) \cite{WuNo17}. Therefore, kernel CCA optimally approximates the
transfer operator singular values and functions within kernel methods, and VAMP is
a suitable optimization method for identifying coherent sets.
\item We propose a new method called \emph{coherent mode decomposition}, which can be seen as a combination of CCA and DMD.
\end{enumerate}

Establishing similar connections between machine learning and dynamical systems theory have previously led to advances in different applications: 
By means of the variational approach of conformation dynamics (VAC) \cite{NoNu13}, optimal estimators for the leading eigenfunctions of reversible time-homogeneous 
transfer operators have been made, which are important to identify metastable sets and rare events \cite{SchuetteFischerHuisingaDeuflhard_JCompPhys151_146,Bovier06:metastability}. This insight has led to the introduction of the TICA method as a way to identify slow collective variables to molecular dynamics  \cite{PPGDN13,SP13}---a key step in the modeling of rare events in molecules. Kernel embeddings of conditional probability distributions\cite{SHSF09,MFSS16} have been related to Perron--Frobenius and Koopman transfer operators~\cite{Ko31, LaMa94} and their eigenvalue decomposition in Ref.~\onlinecite{KSM17}.
In a similar way, optimization of the VAMP score can be used to derive CCA as an optimal linear algorithm to approximate the singular functions of transfer operators~\cite{WuNo17}. Doing the same with neural networks as function approximators leads to VAMPnets~\cite{MPWN:vampnets}. Here we show that a similar connection can be made with kernel methods and transfer operator singular functions and demonstrate that the kernel CCA algorithm approximates these functions. By exploiting that the singular functions are simultaneously the eigenfunctions of the forward-backward dynamics, we can extend this framework to the identification of so-called \emph{coherent sets}---a generalization of metastable sets to nonautonomous and aperiodic systems~\cite{FJ18:coherent}. Coherent sets are regions of the state space that are not dispersed over a specific time interval. That is, if we let the system evolve, elements of a coherent set will, with a high probability, stay close together, whereas other regions of the state space might be distorted entirely. A large number of publications investigate the numerical approximation of coherent sets with other methods, e.g., Refs.~\onlinecite{FrSaMo10:coherent, FrJu15, WRR15, HKTH16:coherent, BK17:coherent, HSD18, FJ18:coherent}, see Ref.~\onlinecite{AP15:review} for an overview of approaches for Lagrangian data. We will not address the problem of possibly sparse or incomplete data. Our goal is to illustrate relationships with established kernel-based approaches and to show that existing methods---developed independently and with different applications in mind, predating many algorithms for the computation of finite-time coherent sets---can be directly applied to detect coherent sets in Lagrangian data.

\begin{figure*}
    \centering
    \includegraphics[width=0.8\textwidth]{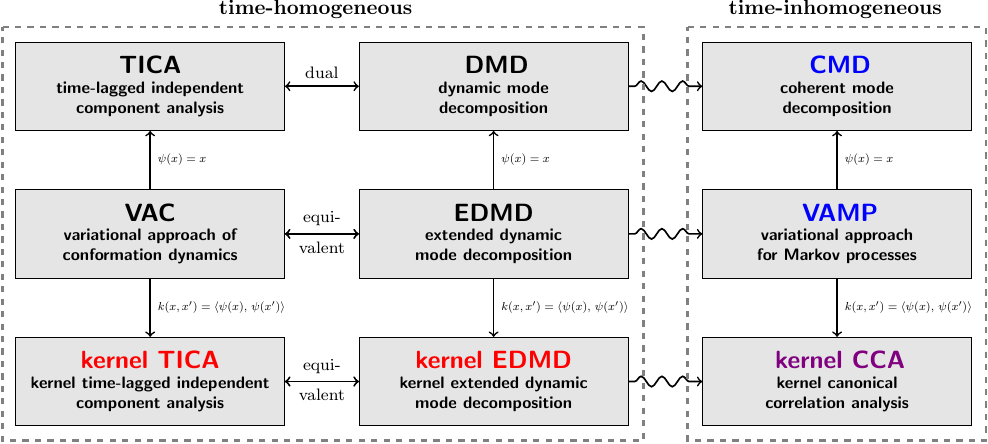}
    \caption{Overview of data-driven methods for the identification of slow collective variables, rare events, or coherent structures.}
    \label{fig:Relationships}
\end{figure*}

A high-level overview of data-driven approaches for the approximation of transfer operators that are relevant for our considerations and relationships with the methods proposed below are shown in Figure~\ref{fig:Relationships}. For the derivations of some of these methods the system is assumed to be reversible and there are other subtle differences, which will not be discussed here.  A comparison of methods for time-homogeneous systems can be found in Ref.~\onlinecite{KNKWKSN18} and extensions to time-inhomogeneous systems in Ref.~\onlinecite{KWNS18:noneq}. Moreover, it was shown that the role played by the eigenfunctions in the time-homogeneous setting (for the detection of metastable sets) is assumed by the left- and right singular functions in the time-inhomogeneous setting (for the detection of coherent sets)\cite{KWNS18:noneq}. The right singular functions encode information about the system at initial time $ t $ and the left singular functions correspond to the system's state at final time $ t + \tau $. The connections between kernel CCA and the singular value decomposition of transfer operators will be described in more detail below.

The remainder of this paper is structured as follows: In Section~\ref{sec:Prerequisites}, we will briefly introduce transfer operators and review the notion of positive definite kernels and induced Hilbert spaces as well as nonlinear generalizations of covariance and cross-covariance matrices. We will then define empirical RKHS operators and show that diverse algorithms can be formulated as eigenvalue problems involving such operators. The relationships between kernel CCA and coherent sets will be studied in Section~\ref{sec:Kernel CCA and coherent sets}. Furthermore, coherent mode decomposition will be derived. Section~\ref{sec:Numerical results} contains numerical results illustrating how to use the presented kernel-based methods for the analysis of dynamical systems. We conclude with a summary of the main results and open problems in Section~\ref{sec:Conclusion}.

\section{Prerequisites}
\label{sec:Prerequisites}

We briefly introduce transfer operators, reproducing kernel Hilbert spaces, and operators mapping from one such space to another one (or itself). For more details on the properties of these spaces and the introduced operators, we refer the reader to Refs.~\onlinecite{Schoe01, Steinwart2008:SVM, SC04:KernelMethods} and Refs.~\onlinecite{Baker70:XCov, Baker1973, KSM17}, respectively.

\subsection{Transfer operators}
\label{ssec:transfer_operators}
Let $ \{ X_t \}_{t \ge 0} $ be a stochastic process defined on the state space $ \inspace \subset \R^d $ and let $ \tau $ be a fixed lag time. We assume that there exists a \emph{transition density function} $ p_\tau \colon \inspace \times \inspace \rightarrow \R $ such that $ p_\tau(y \mid x) $ is the probability of $ X_{t + \tau} = y $ given $ X_t = x $.
For $ 1 \leq r \leq \infty$, let $L^r(\inspace)$ denote the standard space of $r$-Lebesgue integrable functions on $\inspace$. Then, for a probability density $\mu$ on $\inspace$, let $L^r_\mu (\inspace)$ be the spaces of $r$-integrable functions with respect to the corresponding probability measure induced by the density $\mu$; that is, $\norm{f}^r_{L^r_\mu(\inspace)}  = \int |f(x)|^r \ts \mu(x) \ts \dd x$.

Given a probability density $ p \in L^1(\inspace) $ and an observable $ f \in L^\infty(\inspace)$, we define the \emph{Perron--Frobenius operator} $ \pf \colon L^1(\inspace) \to L^1(\inspace) $ and the \emph{Koopman operator} $ \ko \colon L^{\infty}(\inspace) \to L^{\infty}(\inspace) $ by
\begin{align*}
    \left(\pf p \right)(y) &= \int p_\tau(y \mid x) \ts p(x) \ts \dd x, \\
    \left(\ko f\right)(x) &= \int p_\tau(y \mid x) \ts f(y) \ts \dd y.
\end{align*}
Assuming the process admits a unique equilibrium density $ \pi $, i.e., $ \pf \pi = \pi $, we can define for $ u = \pi(x)^{-1} \ts p(x) $ the \emph{Perron--Frobenius operator with respect to the equilibrium density} $ \mathcal{T} \colon L_\pi^1(\inspace) \to L_\pi^1(\inspace) $ as
\begin{equation*}
   \left(\mathcal{T} u\right)(y) = \frac{1}{\pi(y)} \int p_\tau(y \mid x) \ts \pi(x) \ts u(x) \ts \dd x.
\end{equation*}
Under certain conditions, these transfer operators can be defined on $ L^r(\inspace) $ and $L^r_\pi(\inspace)$ for other choices of $r$.
From now on, we will always assume that they are well-defined for $ r = 2 $ (see Refs.~\onlinecite{LaMa94, BaRo95, KKS16} for details). This is common whenever Hilbert space properties are needed in the context of transfer operators.

\begin{remark}
For time-homogeneous systems, the associated transfer operators depend only on the lag time $ \tau $. If the system is time-inhomogeneous, on the other hand, the lag time is not sufficient to parametrize the evolution of the system since it also depends on the starting time. This is described in detail in Ref.~\onlinecite{KWNS18:noneq}. The transition density and the operators thus require two parameters; however, we will omit the starting time dependence for the sake of clarity.
\end{remark}

\subsection{Reproducing kernel Hilbert spaces}

Given a set $ \inspace $ and a space $ \mathbb{H} $ of functions $ f \colon \inspace \to \R $, $ \mathbb{H} $~is called a \emph{reproducing kernel Hilbert space (RKHS)} with inner product $ \innerprod{\cdot}{\cdot}_\mathbb{H} $ if there exists a function $ k \colon \inspace \times \inspace \to \R $ with the following properties:
\begin{enumerate}[label=(\roman*), itemsep=0ex, topsep=1ex]
\item $ \innerprod{f}{k(x, \cdot)}_\mathbb{H} = f(x) $ for all $ f \in \mathbb{H} $, and
\item $ \mathbb{H} = \overline{\mspan\{k(x, \cdot) \mid x \in \inspace \}} $.
\end{enumerate}
The function $ k $ is called a \emph{kernel} and the first property above the \emph{reproducing property}. A direct consequence is that $ \innerprod{k(x, \cdot)}{k(x^\prime, \cdot)}_\mathbb{H} = k(x, x^\prime) $. That is, the map $ \phi \colon \inspace \rightarrow \rkhs $ given by $ x \mapsto k(x, \cdot) $ can be regarded as a feature map associated with $ x $, the so-called \emph{canonical feature map}.\!\footnote{Such a feature map $ \phi \colon \inspace \rightarrow \rkhs $ 
admitting the property $ k(x, x^\prime) = \innerprod{\phi(x)}{\phi(x^\prime)}_\mathbb{H} $
is not uniquely defined. There are other feature space representations such as, for instance, the Mercer feature space.\!\cite{Mercer, Schoe01, Steinwart2008:SVM} As long as we are only interested in kernel evaluations, however, it does not matter which one is considered.}
It is thus possible to represent data by functions in the RKHS. Frequently used kernels include the polynomial kernel and the Gaussian kernel, given by $ k(x, x^\prime) = (c + x^\top x^\prime)^p $ and $ k(x, x^\prime) = \exp(-\|x-x^\prime\|_2^2/2\sigma^2) $, respectively. While the feature space associated with the polynomial kernel is finite-dimensional, the feature space associated with the Gaussian kernel is infinite-dimensional; see, e.g., Ref.~\onlinecite{Steinwart2008:SVM}. Inner products in these spaces, however, are not evaluated explicitly, but only implicitly through kernel evaluations. This is one of the main advantages of kernel-based methods \cite{Scholkopf98:KPCA, SC04:KernelMethods}. Algorithms that can be purely expressed in terms of inner product evaluations can thus be easily \emph{kernelized}, resulting, as described above, in nonlinear extensions of methods such as PCA, CCA, or TICA.

\subsection{Covariance operators and Gram matrices}

Let $(X,Y)$ be a random variable on $ \inspace \times \outspace $, where $ \inspace \subset \R^{d_x} $ and $ \outspace \subset \R^{d_y} $. The dimensions $ d_x $ and $ d_y $ can in principle be different. For our applications, however, the spaces $ \inspace $ and $ \outspace $ are often identical. The associated marginal distributions are denoted by $\pp{P}_x(X)$ and $\pp{P}_y(Y)$, the joint distribution by $\pp{P}(X,Y)$, and the corresponding densities---which we assume exist---by $ p_x(x) $, $ p_y(y) $, and $ p(x, y) $, respectively. Furthermore, let $k$ and $l$ be the kernels associated with $ \inspace $ and $ \outspace$ and $\phi $ and $\psi$ the respective feature maps. We will always assume that requirements such as measurability of the kernels and feature maps as well as separability of the RKHSs are satisfied.\!\footnote{In most cases, these properties follow from mild assumptions about $ \inspace $ and $ \outspace $. For an in-depth discussion of these technical details, see Ref.~\onlinecite{Steinwart2008:SVM}.} The RKHSs induced by the kernels $ k $ and $ l $ are denoted by $ \rkhs[X] $ and $ \rkhs[Y] $.

We will now introduce covariance operators and cross-covariance operators \cite{Baker70:XCov, Baker1973} on RKHSs. In what follows, we will always assume that  $\mathbb{E}_{\scriptscriptstyle \mathit{X}}[k(X,X)] <  \infty$
and $\mathbb{E}_{\scriptscriptstyle \mathit{Y}}[l(Y,Y)] <  \infty$,
which ensures that these operators are well-defined and Hilbert--Schmidt (for a comprehensive overview of kernel covariance operators and their applications, see Ref.~\onlinecite{MFSS16} and references therein). For any $ f \in \rkhs[X] $, let
\begin{equation*}
    \psi(Y) \otimes \phi(X) \colon f \mapsto \psi(Y) \innerprod{\phi(X)}{f}_{\rkhs[X]}
\end{equation*}
denote the \emph{tensor product operator} \cite{Reed} from $ \rkhs[X] $ to $ \rkhs[Y] $ defined by $ \phi(X) $ and $ \psi(Y) $.

\begin{definition}[Covariance operators]
The \emph{covariance operator} $ \cov[XX] \colon \rkhs[X] \to \rkhs[X] $ and the \emph{cross-covariance operator} $ \cov[YX] \colon \rkhs[X] \to \rkhs[Y] $ are defined as
\begin{alignat*}{4}
    \cov[XX] & := \int \phi(X) \otimes \phi(X) \ts \dd \pp{P}(X)
             &&= \mathbb{E}_{\scriptscriptstyle X}[\phi(X) \otimes \phi(X)], \\
    \cov[YX] & := \int \psi(Y) \otimes \phi(X) \ts \dd \pp{P}(Y,X)   
             &&= \mathbb{E}_{\scriptscriptstyle \mathit{YX}}[\psi(Y) \otimes \phi(X)].
\end{alignat*}
\end{definition}
Kernel covariance operators satisfy
\begin{equation*}
    \innerprod{g}{\cov[YX]f}_{\rkhs[Y]} = \mathrm{Cov}[g(Y), f(X)]
\end{equation*}
for all $ f \in \rkhs[X] $, $ g \in \rkhs[Y] $. Defining $ \phi_c(X) = \phi(X) - \mathbb{E}_{\scriptscriptstyle X}[\phi(X)] $ and $ \psi_c(Y) = \psi(Y) - \mathbb{E}_{\scriptscriptstyle Y}[\psi(Y)] $, the corresponding centered counterparts of the covariance and cross-covariance operators $\cov[XX]$ and $\cov[YX]$ are defined in terms of the mean-subtracted feature maps.

As these operators can in general not be determined analytically, empirical estimates are computed from data, i.e.,
\begin{align}
\begin{split} \label{eq:cov_estimates}
    \ecov[XX] &= \frac{1}{n} \sum_{i=1}^n \phi(x_i) \otimes \phi(x_i)
              = \frac{1}{n} \Phi \Phi^\top, \\
    \ecov[YX] &= \frac{1}{n} \sum_{i=1}^n \psi(y_i)\otimes\phi(x_i)
              = \frac{1}{n} \Psi \Phi^\top,
\end{split}
\end{align}
where $ \Phi = [\phi(x_1), \dots, \phi(x_n)] $ and $ \Psi = [\psi(y_1), \dots, \psi(y_n)] $ and the training data $ \{(x_i, y_i)\}_{i=1}^n $ is drawn i.i.d.\ from $ \pp{P}(X, Y) $. Analogously, the mean-subtracted feature maps can be used to obtain empirical estimates of the centered operators.
Since in practice we often cannot explicitly deal with these operators, in particular if the feature space is infinite-dimensional, we seek to reformulate algorithms in terms of Gram matrices.

\begin{definition}[Gram matrices]
Given training data as defined above, the Gram matrices $ \gram[XX], \gram[YY] \in \R^{n \times n} $ are defined as
\begin{align*}
    \gram[XX] &= \Phi^\top \Phi = \big[\ts k(x_i, x_j) \ts\big]_{i,j=1}^n, \\
    \gram[YY] &= \Psi^\top \Psi = \big[\ts l(y_i, y_j) \ts\big]_{i,j=1}^n.
\end{align*}
\end{definition}

For a Gram matrix $ G $, its centered version $ \widetilde{G} $ is defined by $ \widetilde{G} = N_0 \ts G \ts N_0 $, where $ N_0 = \id - \frac{1}{n} \mathds{1} \mathds{1}^\top $ and $ \mathds{1} \in \R^n $ is a vector composed of ones \cite{Bach03:KICA}. Note that centered Gram matrices are not regular.

\begin{remark}
In what follows, if not noted otherwise, we assume that the covariance operators $ \cov[XX] $ and $ \cov[YY] $ and the Gram matrices $ \gram[XX] $ and $ \gram[YY] $ are properly centered for CCA.
\end{remark}

\subsection{Kernel transfer operators}
\label{ssec:Kernel_transfer_operators}

We now show how transfer operators can be written in terms of covariance and cross-covariance operators---this leads to the concept of \emph{kernel transfer operators}. We assume the Perron--Frobenius
operator and the Koopman operator to be well-defined on $ L^2(\inspace) $ as discussed in Section~\ref{ssec:transfer_operators}. Kernel transfer operators follow from the assumption that densities and observables in $ L^2(\inspace) $ can be represented as elements of the RKHS $ \rkhs[X] $. Under some technical requirements, such as $ \int_\inspace k(x,x) \ts \dd x = \int_\inspace \norm{\phi(x)}_{\rkhs[X]}^2 \ts \dd x < \infty$, the elements of $ \rkhs[X] $ are included in $ L^2(\inspace) $  when they are identified with the respective equivalence class of square integrable functions. This correspondence can be derived from the theory of $ L^2(\inspace) $ integral operators\cite{Steinwart2008:SVM} and is often used in statistical learning theory~\cite{RBD10}. We may therefore assume that we can identify RKHS elements with the corresponding equivalence classes of functions in $L^2(\inspace)$. By requiring $\mathbb{E}_{\mu}[k(X,X)] <  \infty$ for a probability density $\mu(x)$, we obtain a similar statement for $L^2_\mu(\inspace)$.

We refer to Ref.~\onlinecite{KSM17} for the derivation of kernel transfer operators and a description of their relationships with kernel embeddings of conditional distributions. We will omit the technical details and directly define kernel transfer operators as the RKHS analogue of the standard transfer operators defined in Section~\ref{ssec:transfer_operators}. Using the same integral representations as before and defining the transfer operators on $\rkhs[X]$ instead of $L^2(\inspace)$, we obtain the \emph{kernel Perron--Frobenius operator} $\pf[k] \colon\rkhs[X] \rightarrow \rkhs[X] $ and the \emph{kernel Koopman operator} $ \ko[k] \colon \rkhs[X] \rightarrow \rkhs[X] $, respectively.

By defining the time-lagged process $Y_t = X_{t + \tau}$, we can write kernel transfer operators in terms of covariance and cross-covariance operators~\cite{KSM17}. Note that $X_t$ and $Y_t$ are defined on the same state space $ \inspace $; therefore, we have $\rkhs[X] = \rkhs[Y] $ and hence $ \cov[YX] : \rkhs[X] \rightarrow \rkhs[X] $ in this special case.
We obtain the important properties $ \cov[XX] \pf[k] \ts g = \cov[YX] g $
and $ \cov[XX] \ko[k] \ts g = \cov[XY] g $
for all $ g \in \rkhs[X] $, which allows us to write
\begin{equation} \label{eq:KTOs}
\begin{split}
    \pf[k] &= (\cov[XX] + \varepsilon \idop)^{-1} \ts \cov[YX], \\
    \ko[k] &= (\cov[XX] + \varepsilon \idop)^{-1} \ts \cov[XY].
\end{split}
\end{equation}
Here, $ (\cov[XX] + \varepsilon \idop)^{-1} $ is the Tikhonov-regularized
inverse of $ \cov[XX] $ with regularization parameter $\varepsilon > 0$.\!\footnote{See Refs.~\onlinecite{Gr93, EG96, EHN96} for a detailed discussion of ill-posed inverse problems and the regularization of bounded linear operators on Hilbert spaces.} Note the abuse of notation, since equality in the above inverse problems is only given asymptotically for $\varepsilon \to 0$ and pointwise for feasible $ \cov[YX] g \in \rkhs[X]$. Since $ \cov[XX] $ is a compact operator,
it does not admit a globally defined bounded inverse if the RKHS is infinite-dimensional. However, $ (\cov[XX] + \varepsilon \idop)^{-1} $ always  exists and is bounded. In fact, the operators $ \pf[k] $ and $ \ko[k] $ as given in the regularized form above are Hilbert--Schmidt.

The above notation and regularization of inverse covariance operators is standard in the context of kernel embeddings of conditional distributions and related Bayesian learning techniques. We refer to Refs.~\onlinecite{SHSF09, Song2013, Fukumizu13:KBR, Fukumizu15:NBI, MFSS16} for detailed discussions of properties of this ill-posed inverse problem in specific applications.

By replacing the analytical covariance operators with their empirical estimates in~\eqref{eq:KTOs}, we obtain empirical estimates for kernel transfer operators \cite{KSM17}. 
As done with empirical covariance operators in~\eqref{eq:cov_estimates}, it is possible to rewrite the empirical estimates of kernel transfer operators in terms of RKHS features in $ \Phi $ and $\Psi$ 
(see Refs.~\onlinecite{MFSS16, KSM17} for the derivation):
\begin{align}
\begin{split} \label{eq:KTO_estimates}
    \widehat{\pf[k]} &= (\ecov[XX] + \varepsilon \idop)^{-1} \ecov[YX] \\ &= \Psi \ts \big(\gram[XY]^{-1} \ts (\gram[XX] + n \varepsilon \id)^{-1} \ts \gram[XY]\big) \ts \Phi^\top, \\
    \widehat{\ko[k]} &= (\ecov[XX] + \varepsilon \idop)^{-1} \ecov[XY] \\ &= \Phi \ts (\gram[XX] + n \varepsilon \id)^{-1} \ts \Psi^\top.
\end{split}
\end{align}
In this case, $\Phi$ and $\Psi$ both contain observations in the same space $\rkhs[X]$, since $X_t$ and $Y_t$ are both defined on $\inspace$.

\subsection{Empirical RKHS operators}

In what follows, we will consider finite-rank RKHS operators given by a matrix which represents the action of the operator on fixed elements in the RKHSs. We will use this general setting to formulate results about the eigenvalues and eigenfunctions of empirical RKHS operators. Given a matrix $ B \in \R^{n \times n} $, we define  the bounded finite-rank operator $ \widehat{\mathcal{S}} \colon \rkhs[X] \to \rkhs[Y] $ by
\begin{equation} \label{eq:empirical_operator}
    \widehat{\mathcal{S}} = \Psi B \Phi^\top = \sum_{i,j=1}^n b_{ij} \ts \psi(y_i) \otimes \phi(x_j).
\end{equation}
We remark that although $ \Psi $ and $ \Phi $ may contain infinite-dimensional objects, we express inner products between RKHS elements in the classical matrix-vector multiplication form. That is, we interpret the embedded RKHS elements as (potentially infinite-dimensional) column vectors. This notation has become a de-facto standard in the machine learning community~\cite{MFSS16}. We can write empirical estimates of covariance operators in the form of~\eqref{eq:empirical_operator}. If the RKHS training features in $\Phi$ and $ \Psi $ are generated i.i.d.\ by the joint probability distribution $ \mathbb{P}(X,Y) $ of random variables $ X $ and $ Y $, then the cross-covariance operator $ \ecov[YX] $ takes the general form of an empirical RKHS operator with $ B = \frac{1}{n} \id $. We obtain $ \ecov[XX] $ as another special case with identical features $\Psi = \Phi $ drawn only from $ \mathbb{P}(X) $. Furthermore, the empirical estimates of the kernel Perron--Frobenius and kernel Koopman operator are special cases of $ \widehat{\mathcal{S}} $ as seen in~\eqref{eq:KTO_estimates} with $ B = \gram[XY]^{-1} \ts (\gram[XX] + n \varepsilon \id)^{-1} \ts \gram[XY] $ and $B = (\gram[XX] + n \varepsilon \id)^{-1} $, respectively. Note that the roles of $ \Phi $ and $ \Psi $ are interchanged for the empirical estimate of the Koopman operator, i.e., it is of the form $ \widehat{\mathcal{S}} = \Phi B \Psi^\top$.

We now show how spectral decomposition techniques can be applied to empirical RKHS operators in this general setting.\!\footnote{In general, all considered kernel transfer operators in this paper are compositions of compact and bounded operators and therefore compact. They admit series representations in terms of singular value decompositions as well as eigendecompositions in the self-adjoint case\cite{Reed}. The functional analytic details and the convergence of $ \widehat{\mathcal{S}} $ and its spectral properties in the infinite-data limit depend on the specific scenario and are beyond the scope of this paper.}  We can compute eigenvalues and corresponding eigenfunctions of $ \widehat{\mathcal{S}} $ by solving auxiliary matrix eigenvalue problems.
For the sake of self-containedness, we briefly reproduce the eigendecomposition result from Ref.~\onlinecite{KSM17}.

\begin{proposition} \label{prop:eigenfunctions}
Suppose $ \Phi $ and $ \Psi $ contain linearly independent elements. Let $ \widehat{\mathcal{S}} = \Psi B \Phi^\top $, then
\begin{enumerate}[wide, label=(\roman*), itemindent=\parindent, itemsep=0ex, topsep=0.5ex]
\item $ \widehat{\mathcal{S}} $ has an eigenvalue $ \lambda \ne 0 $ with corresponding eigenfunction $ \varphi = \Psi v $ if and only if $ v $ is an eigenvector of $ B \ts \gram[XY] $ associated with $ \lambda $, and, similarly,
\item $ \widehat{\mathcal{S}} $ has an eigenvalue $ \lambda \ne 0 $ with corresponding eigenfunction $ \varphi = \Phi \ts \gram[XX]^{-1} \ts v $ if and only if $ v $ is an eigenvector of $ \gram[XY] \ts B $.
\end{enumerate}
\end{proposition}

For the Gaussian kernel, linear independence of elements in $ \Phi $ and $ \Psi $ reduces to requiring that the training data contains pairwise distinct elements in $ \inspace $ and $ \outspace $, respectively.
For dynamical systems applications, we typically assume that $ \Phi $ and $ \Psi $ contain information about the system at time $ t $ and at time $ t + \tau $, respectively. A more detailed version of Proposition~\ref{prop:eigenfunctions} and its extension to the singular value decomposition are described in Ref.~\onlinecite{MSKS18}. Further properties of $ \widehat{\mathcal{S}} $ and its decompositions will be studied in future work. Note that we generally assume that empirical estimates of RKHS operators converge in probability to their analytical counterparts in operator norm in the infinite data limit. These statistical properties and the resulting associated spectral convergence are examined in for example in Ref.~\onlinecite{RBD10}.

\subsection{Applications of RKHS operators}

Decompositions of RKHS operators have diverse applications, which we will only touch upon here. We will consider a specific problem---namely, kernel CCA---in Section~\ref{sec:Kernel CCA and coherent sets}.

\begin{enumerate}[wide, label=(\alph*), itemindent=\parindent, itemsep=0ex, topsep=0.5ex]
\item By sampling points from the uniform distribution, the \emph{Mercer feature map}\cite{Mercer, Schoe01, Steinwart2008:SVM} with respect to the Lebesgue measure on $ \inspace $ can be approximated by computing eigenfunctions of $ \ecov[XX] $---i.e., $ B = \frac{1}{n} \id $ and the auxiliary matrix eigenvalue problem is $ \frac{1}{n} \gram[XX] \ts v = \lambda \ts v $---as shown in Ref.~\onlinecite{MSKS18}. This can be easily extended to other measures.
\item Similarly, given an arbitrary data set $ \{x_i\}_{i=1}^n $, \emph{kernel PCA} computes the eigenvectors corresponding to the largest eigenvalues of the centered Gram matrix $ \gram[XX] $ and defines these eigenvectors as the data points projected onto the respective principal components. It is well-known that kernel PCA can also be defined in terms of the centered covariance operator $ \ecov[XX] $. A detailed connection of the spectrum of the Gram matrix and the covariance operator is given in Ref.~\onlinecite{STWCK02}.
Up to scaling, the eigenfunctions evaluated in the data points correspond to the principal components.

\item Given training data $ x_i \sim p_x $ and $ y_i = \Theta^\tau(x_i) $, where $ \Theta $ denotes the flow associated with the dynamical system and $ \tau $ the lag time---that is, if $ x_i $ is the state of the system at time $ t $, then $ y_i $ is the state of the system at time $ t + \tau $---, we define $ \Phi $ and $ \Psi $ as above. Eigenvalues and eigenfunctions of kernel transfer operators can be computed by solving a standard matrix eigenvalue problem (see Proposition~\ref{prop:eigenfunctions}).
Eigendecompositions of these operators result in metastable sets. For more details and real-world examples, see Refs.~\onlinecite{KSM17, KBSS18}. The main goal of this paper is the extension of the aforementioned methods to compute \emph{coherent sets} instead of \emph{metastable sets}.
\end{enumerate}

\section{Kernel CCA and coherent sets}
\label{sec:Kernel CCA and coherent sets}

This section contains the main results of our paper. We derive kernel CCA \cite{MRB01:CCA} for finite and infinite-dimensional feature spaces from the viewpoint of dynamical systems, and show that kernel CCA can be used to approximate coherent sets in dynamical data. Furthermore, we derive the new coherent mode decomposition method.

Given two multidimensional random variables $ X $ and $ Y\! $, standard CCA finds two sets of basis vectors such that the correlations between the projections of $ X $ and $ Y $ onto these basis vectors are maximized \cite{Borga01:CCA}. The new bases can be found by computing the dominant eigenvalues and corresponding eigenvectors of a matrix composed of covariance and cross-covariance matrices. Just like kernel PCA is a nonlinear extension of PCA, kernel CCA is a generalization of CCA. The goal of kernel CCA is to find two \emph{nonlinear} mappings $ f(X) $ and $ g(Y) $, where $ f \in \rkhs[X] $ and $ g \in \rkhs[Y] $, such that their correlation is maximized \cite{Fukumizu07:KCCA}. That is, instead of matrices, kernel CCA is now formulated in terms of covariance and cross-covariance operators. More precisely, the kernel CCA problem can be written as
\begin{equation*}
    \sup_{\substack{ f \in \rkhs[X] \\ g \in \rkhs[Y] }} \innerprod{g}{\cov[YX] f}_{\rkhs[Y]} \quad \text{s.t.} \quad
    \begin{cases}
    \innerprod{f}{\cov[XX]f}_{\rkhs[X]} = 1, \\
    \innerprod{g}{\cov[YY]g}_{\rkhs[Y]} = 1,
    \end{cases}
\end{equation*}
and the solution is given by the eigenfunctions corresponding to the largest eigenvalue of the problem
\begin{equation} \label{eq:KCCA-eig-full}
    \begin{cases}
        \cov[YX] f = \rho \ts \cov[YY] \ts g, \\
        \cov[XY] g = \rho \ts \cov[XX] \ts f.
    \end{cases}
\end{equation}
Further eigenfunctions corresponding to subsequent eigenvalues can be taken into account as in the standard setting described above. In practice, the eigenfunctions are estimated from finite samples. The empirical estimates of $ f $ and $ g $ are denoted by $ \widehat{f} $ and $ \widehat{g} $, respectively.

\begin{example}
In order to illustrate kernel CCA, let us analyze a synthetic data set similar to the one described in Ref.~\onlinecite{Fukumizu07:KCCA} using a Gaussian kernel with bandwidth $ \sigma = 0.3 $. Algorithms to solve the CCA problem will be described below. The results are shown in Figure~\ref{fig:KernelCCA}. Classical CCA would not be able to capture the nonlinear relationship between $ X $ and~$ Y $.  \exampleSymbol

\begin{figure*}
    \begin{minipage}{0.245\textwidth}
        \centering
        \subfiguretitle{\hspace{1.9em} (a)} \vspace*{0.3ex}
        \includegraphics[width=\textwidth]{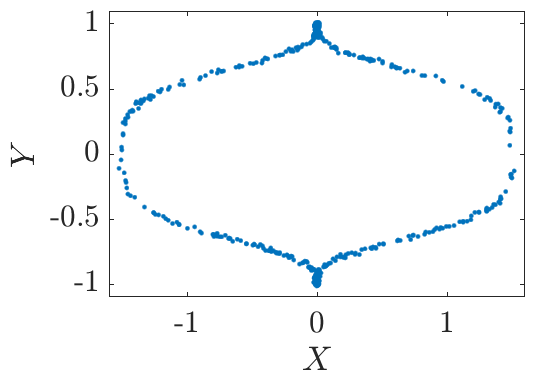}
    \end{minipage}
    \begin{minipage}{0.24\textwidth}
        \centering
        \subfiguretitle{\hspace{1.9em} (b)} \vspace*{-0.6ex}
        \includegraphics[width=\textwidth]{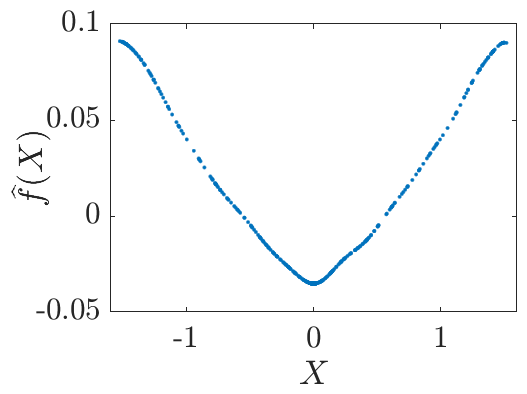}
    \end{minipage}
    \begin{minipage}{0.24\textwidth}
        \centering
        \subfiguretitle{\hspace{1.9em} (c)} \vspace*{-0.6ex}
        \includegraphics[width=\textwidth]{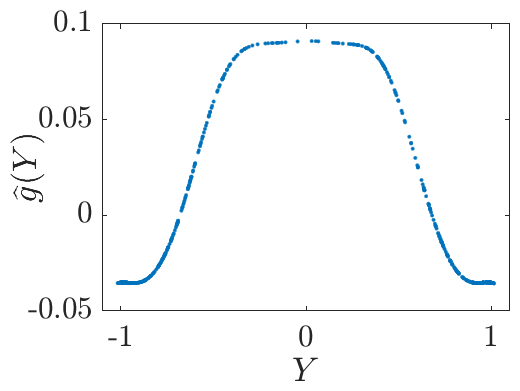}
    \end{minipage}
    \begin{minipage}{0.24\textwidth}
        \centering
        \subfiguretitle{\hspace{1.6em} (d)} \vspace*{-0.6ex}
        \includegraphics[width=\textwidth]{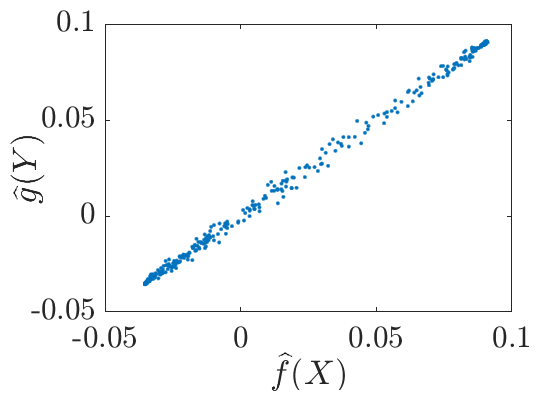}
    \end{minipage}
    \caption{Kernel CCA applied to noisy generalized superellipse data. The transformed variables $ \widehat{f}(X) $ and $ \widehat{g}(Y) $ are clearly highly correlated.}
    \label{fig:KernelCCA}
\end{figure*}

\end{example}

\subsection{RKHS operator formulation}
\label{ssec:RKHS operator formulation}

Since the inverses of the covariance operators in general do not exist, the regularized versions $(\cov[XX] + \varepsilon \idop)^{-1}$ and $(\cov[YY] + \varepsilon \idop)^{-1}$ (cf.~Section~\ref{ssec:Kernel_transfer_operators}) are also typically used in the context of CCA\cite{Fukumizu07:KCCA}. Solving the first equation in \eqref{eq:KCCA-eig-full} for $ g $ and inserting it into the second equation, this results in
\begin{equation} \label{eq:KCCA-eig-reduced}
    \quad \big(\cov[XX] + \varepsilon \idop\big)^{-1} \cov[XY] \big(\cov[YY] + \varepsilon \idop\big)^{-1} \cov[YX] f = \rho^2 f.
\end{equation}
Comparing this with the aforementioned transfer operator representations \eqref{eq:KTOs}, $ (\cov[XX] + \varepsilon \idop)^{-1} \cov[XY] $ can be interpreted as an approximation of the kernel Koopman operator, and $ (\cov[YY] + \varepsilon \idop)^{-1} \cov[YX] $ as a kernel Koopman operator where now the roles of $ X $ and $ Y $ are reversed or as a reweighted Perron--Frobenius operator. The composition of these operators corresponds to a push-forward and subsequent pull-back of a density $ f $. Eigenfunctions of the operator whose associated eigenvalues are close to one thus remain nearly unchanged under the forward-backward dynamics. This is closely related to the notion of \emph{coherence} as introduced in Refs.~\onlinecite{FrSaMo10:coherent, Froyland13:coherent} and will be discussed in Section~\ref{ssec:Relationships between kernel CCA and transfer operators}.

\begin{lemma}
Replacing the covariance and cross-covariance operators by their empirical estimates, the eigenvalue problem~\eqref{eq:KCCA-eig-reduced} can be written as
\begin{equation*}
    \Phi B \Phi^\top \widehat{f} = \rho^2 \widehat{f},
\end{equation*}
with $ B = (\gram[XX] + n \varepsilon \id)^{-1} (\gram[YY] + n \varepsilon \id)^{-1} \gram[YY] $.
\end{lemma}

\begin{proof}
Inserting the definitions of the empirical covariance and cross-covariance operators yields
\begin{equation*}
    \big(\Phi \Phi^\top + n \varepsilon \idop\big)^{-1} \Phi \Psi^\top \big(\Psi \Psi^\top + n \varepsilon \idop\big)^{-1} \Psi \Phi^\top \widehat{f} = \rho^2 \widehat{f}.
\end{equation*}
Using $ \Psi^\top \left(\Psi\Psi^\top + n \varepsilon \idop\right)^{-1} = \left( \Psi^\top \Psi + n \varepsilon \id\right)^{-1} \Psi^\top $, see Ref.~\onlinecite{MFSS16}, and a similar identity for $ \Phi $ concludes the proof.
\end{proof}

That is, the empirical RKHS operator for kernel CCA is of the form $ \widehat{S} = \Phi B \Phi^\top $. Applying Proposition~\ref{prop:eigenfunctions}, we must solve the auxiliary problem
\begin{enumerate}[label=(\roman*), itemsep=0ex, topsep=1.5ex]
\item $ (\gram[XX] + n \varepsilon \id)^{-1} (\gram[YY] + n \varepsilon \id)^{-1} \gram[YY] \ts \gram[XX] \ts v = \rho^2 \ts v $, with $ \widehat{f} = \Phi \ts v $, or
\item $ \gram[XX] (\gram[XX] + n \varepsilon \id)^{-1} (\gram[YY] + n \varepsilon \id)^{-1} \gram[YY] \ts v = \rho^2 \ts v $, with $ \widehat{f} = \Phi \ts (\gram[XX] + n \varepsilon \id)^{-1} \ts v $.
\end{enumerate}
Since $\gram[XX]$ and $(\gram[XX] + n \varepsilon \id)^{-1}$ as well as $\gram[YY]$ and $(\gram[YY] + n \varepsilon \id)^{-1}$ commute, the first problem can be equivalently rewritten as $ (\gram[XX] + n \varepsilon \id)^{-1} \gram[YY] (\gram[YY] + n \varepsilon \id)^{-1} \gram[XX] \ts v = \rho^2 \ts v $ and the second as $ (\gram[XX] + n \varepsilon \id)^{-1} \gram[XX] \gram[YY] (\gram[YY] + n \varepsilon \id)^{-1} \ts v = \rho^2 \ts v $. The eigenfunction associated with the largest eigenvalue solves the CCA problem, but in order to detect coherent sets, we will need more eigenfunctions later. To obtain the function $ g $ corresponding to $ \rho $, we compute
\begin{enumerate}[label=(\roman*), itemsep=0ex, topsep=1.5ex]
\item $ \widehat{g} = \frac{1}{\rho} \Psi (\gram[YY] + n \varepsilon \id)^{-1} \gram[XX] v $, or
\item $ \widehat{g} = \frac{1}{\rho} \Psi (\gram[YY] + n \varepsilon \id)^{-1} \gram[XX] (\gram[XX] + n \ts \varepsilon \id)^{-1} v $.
\end{enumerate}

\begin{mdframed}[backgroundcolor=boxback,hidealllines=true]
\begin{textalgorithm} \label{alg:CCA}
The CCA problem can be solved as follows:
\begin{enumerate}[leftmargin=3ex, itemindent=0ex, itemsep=0ex, topsep=0.5ex]
\item Choose a kernel $ k $ and regularization $ \varepsilon $.
\item Compute the centered gram matrices $ \gram[XX] $ and $ \gram[YY] $.
\item Solve $ \gram[XX] (\gram[XX] + n \varepsilon \id)^{-1} (\gram[YY] + n \varepsilon \id)^{-1} \gram[YY] \ts v = \rho^2 \ts v $.
\end{enumerate}
\vspace*{1ex}
\end{textalgorithm}
\end{mdframed}

The corresponding eigenfunction $ \widehat{f} $ evaluated at all data points $ x_1, \dots, x_n $, denoted by $ \widehat{f}_X $, is then approximately given by the vector $ v $. We can evaluate the eigenfunctions at any other point as described above, but we will mainly use the eigenfunction evaluations at the sampled data points for clustering into coherent sets.

Algorithm \ref{alg:CCA} is based on the second problem formulation, i.e., item (ii) above. However, the first variant can be used in the same way. Alternatively, we can rewrite it as an eigenvalue problem of the form
\begin{equation*}
    \begin{cases}
        (\gram[YY] + n \varepsilon \id)^{-1} \gram[XX] \ts v = \rho \ts w, \\
        (\gram[XX] + n \varepsilon \id)^{-1} \gram[YY] \ts w = \rho \ts v,
    \end{cases}
\end{equation*}
and, consequently,
\begin{equation} \label{eq:KCCA-gen-eig}
    \setlength\arraycolsep{0pt}
    \begin{bmatrix}
        0 & \gram[YY] \\
        \gram[XX] & 0
    \end{bmatrix}
    \hspace{-3pt}
    \begin{bmatrix}
        v \\ w
    \end{bmatrix}
    \!
    = \rho
    \!
    \begin{bmatrix}
        (\gram[XX] + n \varepsilon \id) & 0 \\
        0 & (\gram[YY] + n \varepsilon \id)
    \end{bmatrix}
    \hspace{-3pt}
    \begin{bmatrix}
        v \\ w
    \end{bmatrix}.
\end{equation}
Other formulations can be derived in a similar fashion. The advantage is that no matrices have to be inverted. However, the size of the eigenvalue problem doubles, which might be problematic if the number data points $n$ is large.

\begin{remark}
In order to apply the algorithms, we first need to choose a kernel and then tune its parameters, e.g., the bandwidth $ \sigma $ of the Gaussian kernel, and also the regularization parameter $ \varepsilon $. If the bandwidth is too small, this leads to overfitting and to oversmoothing if it is too large. Cross-validation techniques can be used to select suitable hyperparameters. The kernel itself determines the complexity of the function space in which the eigenfunctions are sought (see Section~\ref{ssec:Relationships between kernel CCA and transfer operators}). Additionally, the results depend on the lag time~$ \tau $. Sets that are coherent for a given lag time are not necessarily coherent for a different lag time since these sets might be dispersed again.
\end{remark}

The generalized eigenvalue problem \eqref{eq:KCCA-gen-eig} is almost identical to the one derived in Ref.~\onlinecite{Bach03:KICA}, with the difference that regularization is applied in a slightly different way. That is, the direct eigendecomposition of RKHS operators as proposed in Ref.~\onlinecite{KSM17} results, as expected, in variants of kernel CCA. The statistical convergence of kernel CCA, showing that finite sample estimators converge to the corresponding population counterparts, has been established in Ref.~\onlinecite{Fukumizu07:KCCA}. Kernel CCA can be extended to more than two variables or views of the data as described in Refs.~\onlinecite{Bach03:KICA, SC04:KernelMethods}, which might also have relevant applications in the dynamical systems context.

\subsection{Finite-dimensional feature space}
\label{ssec:Finite-dimensional feature space}

If the state spaces of the kernels $ k $ and $ l $ are finite-dimensional, we can directly solve the eigenvalue problem \eqref{eq:KCCA-eig-full} or \eqref{eq:KCCA-eig-reduced}. Assuming the feature space of the kernel $ k $ is $ r_x $-dimensional and spanned by the basis functions $ \{ \phi_1, \dots, \phi_{r_x} \} $, we define $ \phi \colon \inspace \to \R^{r_x} $ by $ \phi(x) = [\phi_1(x), \dots, \phi_{r_x}(x)]^\top $. That is, we are now using an explicit feature space representation. This induces a kernel by defining $ k(x, x^\prime) = \innerprod{\phi(x)}{\phi(x^\prime)} $.\!\footnote{For the Mercer feature space representation\cite{Mercer, Schoe01} the functions form an orthogonal basis, but orthogonality is not required here.}
We could, for instance, select a set of radial basis functions, monomials, or trigonometric functions. Analogously, we define a vector-valued function $ \psi \colon \outspace \to \R^{r_y} $, with $ \psi(y) = [\psi_1(y), \dots, \psi_{r_y}(y)]^\top $, where $ r_y $ is the dimension of the feature space of the kernel $ l $. Any function in the respective RKHS can be written as $ f = \alpha^\top \phi $ and $ g = \beta^\top \psi $, where $ \alpha \in \R^{r_x} $ and $ \beta \in \R^{r_y} $ are coefficient vectors.

Given training data $ \{(x_i, y_i)\}_{i=1}^n $ drawn from the joint probability distribution, we obtain $ \Phi \in \R^{r_x \times n} $ and $ \Psi \in \R^{r_y \times n} $ and can compute the centered covariance and cross-covariance matrices $ \ecov[XX] $, $ \ecov[XY] $, and $ \ecov[YY] $ explicitly.

\begin{mdframed}[backgroundcolor=boxback,hidealllines=true]
\begin{textalgorithm} \label{alg:CCA explicit}
Given explicit feature maps, we obtain the following CCA algorithm:
\begin{enumerate}[leftmargin=3ex, itemindent=0ex, itemsep=0ex, topsep=0.5ex]
\item Select basis functions $ \phi $ and $ \psi $ and regularization~$ \varepsilon $.
\item Compute (cross-)covariance matrices $ \ecov[XX] $, $ \ecov[XY] $, $ \ecov[YX] $, and $ \ecov[YX] $.
\item Solve the eigenvalue problem \\ $ \big(\ecov[XX] + \varepsilon \idop\big)^{-1} \ecov[XY] \big(\ecov[YY] + \varepsilon \idop\big)^{-1} \ecov[YX] v = \rho^2 \ts v $.
\end{enumerate}
\vspace*{1ex}
\end{textalgorithm}
\end{mdframed}

The eigenfunctions are then given by $ \widehat{f}(x) = \innerprod{v}{\phi(x)} $. Expressions for $ \widehat{g} $ can be derived analogously.

The difference between the Gram matrix approach described in Section~\ref{ssec:RKHS operator formulation} and the algorithm proposed here is that the size of the eigenvalue problem associated with the former depends on the number of data points and permits the dimension of the feature space to be infinite-dimensional, whereas the eigenvalue problem associated with the latter depends on the dimension of the feature space but not on the size of the training data set. This is equivalent to the distinction between \emph{extended dynamic mode decomposition} (EDMD) \cite{WKR15} and kernel EDMD \cite{WRK15} (or the variational approach \cite{NoNu13} and kernel TICA \cite{SP15}, where the system is typically assumed to be reversible; see Ref.~\onlinecite{KSM17} for a detailed comparison) with the small difference that often the Moore--Penrose pseudoinverse~\cite{Penrose} is used for EDMD in lieu of the Tikhonov-regularized inverse.

\subsection{Relationships with VAMP}
\label{ssec:Relationships}

Defining $ v = \big(\ecov[XX] + \varepsilon \idop\big)^{\!\sfrac{-1}{2}} \ts \widetilde{v} $, the eigenvalue problem in Algorithm~\ref{alg:CCA explicit} becomes
\begin{equation*}
    \big(\ecov[XX] + \varepsilon \idop\big)^{\!\sfrac{-1}{2}} \ts \ecov[XY] \ts \big(\ecov[YY] + \varepsilon \idop\big)^{-1} \ts \ecov[YX] \ts \big(\ecov[XX] + \varepsilon \idop\big)^{\!\sfrac{-1}{2}} \, \widetilde{v} = \rho^2 \ts \widetilde{v}.
\end{equation*}
The transformed eigenvectors $ \widetilde{v} $ are thus equivalent to the right singular vectors of the matrix
\begin{equation} \label{eq:VAMP_representation}
    \big(\ecov[YY] + \varepsilon \idop\big)^{\sfrac{-1}{2}} \ts \ecov[YX] \ts \big(\ecov[XX] + \varepsilon \idop\big)^{\!\sfrac{-1}{2}}
\end{equation}
and the values $ \rho $ are given by the singular values, which we assume to be sorted in nonincreasing order. 

This form is the kernel version of the TCCA method that has been first derived as a way to approximate transfer operator singular functions using VAMP, see Ref.~\onlinecite{WuNo17}.
VAMP is an optimization principle that defines a score function whose optimum leads to specific data-driven algorithms. The VAMP-$r$ score is defined as
\begin{equation*}
    S(X,Y) = \sum_i^k \rho^r(X,Y)
\end{equation*}
where $\rho^r(X,Y)$ are the singular value estimates obtained from an SVD of, e.g., \eqref{eq:VAMP_representation}, and $r$ is a positive integer.

When using $r=1$ and kernel feature functions, we obtain the kernel CCA algorithms
\ref{alg:CCA} or \ref{alg:CCA explicit}. When doing the same on an explicit
basis set of feature functions, we obtain TCCA, i.e., time-lagged CCA \cite{WuNo17, No18}. However, since with VAMP a score (or loss) function is available, we have now turned the identification of coherent sets into a generic machine learning problem. For example, training neural networks with VAMP results in VAMPnets, a deep learning method to low-rank approximattion of the transfer operators and the identification of metastable or coherent sets.

On the other hand, the fact that kernel CCA results from maximizing the VAMP score within a kernel approach shows that we can use the VAMP score in the context of cross-validation in order to optimally determine hyperparameters such as the kernel function. Furthermore we can explore other choices as $r=1$. For example, the choice $r=2$ has an interesting interpretation in terms of \textit{kinetic maps}, which are embeddings of the dominant eigenspace or singular space of a transfer operator where Euclidean distances are related to timescales of transitions~\cite{NoeClementi_JCTC15_KineticMap}.

\subsection{Relationships between kernel CCA and transfer operators}
\label{ssec:Relationships between kernel CCA and transfer operators}

We have seen in Section~\ref{ssec:RKHS operator formulation} that the resulting eigenvalue problem \eqref{eq:KCCA-eig-reduced} involves expressions resembling kernel transfer operators. The goal now is illustrate how this eigenvalue problem is related to the operators derived in Ref.~\onlinecite{BK17:coherent} for detecting coherent sets. We first introduce a \emph{forward operator} $ \mathcal{F} \colon L_\mu^2(\inspace) \to L^2(\outspace) $ by
\begin{equation*}
    (\mathcal{F} f)(y) = \int p_\tau(y \mid x) \ts f(x) \ts \mu(x) \ts \dd x,
\end{equation*}
where $ \mu $ is some reference density of interest. Furthermore, let $ \nu = \mathcal{F} \mathds{1} $ be the image density obtained by mapping the indicator function on $ \inspace $ forward in time. Normalizing $ \mathcal{F} $ with respect to $ \nu $, we obtain a new operator $ \mathcal{A} \colon L_\mu^2(\inspace) \to L_\nu^2(\outspace) $ and its adjoint $ \mathcal{A}^* \colon L_\nu^2(\outspace) \to L_\mu^2(\inspace) $, with
\begin{align*}
    (\mathcal{A} f)(y) &= \int \frac{p_\tau(y \mid x)}{\nu(y)} f(x) \ts \mu(x) \ts \dd x, \\
    (\mathcal{A}^* g)(x) &= \int p_\tau(y \mid x) \ts g(y) \ts \dd y.
\end{align*}
It holds that $ \innerprod{\mathcal{A} f}{g}_\nu = \innerprod{f}{\mathcal{A}^* g}_\mu $. Consequently, $ \mathcal{A} $ plays the role of a reweighted Perron--Frobenius operator, whereas $ \mathcal{A}^* $ can be interpreted as an analogue of the Koopman operator
(note that $\mathcal{A}$ and $\mathcal{A}^*$ are defined on reweighted $L^2$-spaces). A more detailed derivation can be found in Ref.~\onlinecite{BK17:coherent}, where the operator $ \mathcal{A}^* \mathcal{A} $ (or a trajectory-averaged version thereof) is used to detect coherent sets. We want to show that this is, up to regularization, equivalent to the operator in \eqref{eq:KCCA-eig-reduced}.

\begin{proposition} \label{prop:reweighted_to}
Assuming that $ \mathcal{A} f \in \rkhs[Y] $ for all $ f \in \rkhs[X] $, it holds that $ \cov[YY] \mathcal{A} f = \cov[YX]f $.
\end{proposition}

\begin{proof}
The proof is almost identical to the proof for the standard Perron--Frobenius operator (see Ref.~\onlinecite{KSM17}). For all $g \in \rkhs[Y]$, we obtain
\begin{align*}
    \innerprod{\cov[YY] \mathcal{A} f}{g}_{\rkhs[Y]}
    &= \mathbb{E}_{\scriptscriptstyle Y}[\mathcal{A} f(Y) \ts g(Y)] \\
    &= \iint \frac{p(y \mid x)}{\nu(y)} f(x) \ts \mu(x) \ts \dd x \ts g(y) \ts \nu(y) \dd y \\
    &= \iint p(y \mid x) \ts f(x) \ts g(y) \ts \mu(x) \ts \dd x \ts \dd y \\
    &= \iint p(x, y) \ts f(x) \ts g(y) \ts \dd x \ts \dd y \\
    &= \mathbb{E}_{\scriptscriptstyle XY}[f(X) \ts g(Y)] \\
    &= \innerprod{\cov[YX] f}{g}_{\rkhs[Y]}. \qedhere
\end{align*}
\end{proof}

We define the RKHS approximation of the operator $ \mathcal{A} $ by $ \mathcal{A}_k = (\cov[YY] + \varepsilon \idop)^{-1} \cov[YX] $. Note that the operator technically depends not only on $ k $ but also on $ l $, which we omit for brevity. In practice, we typically use the same kernel for $ \inspace $ and $ \outspace $. As a result, the eigenvalue problem~\eqref{eq:KCCA-eig-reduced} can now be written as
\begin{equation*}
    \ko[k] \mathcal{A}_k f = \rho^2 f.
\end{equation*}
The adjointness property for $ \varepsilon = 0 $, i.e., assuming that the inverse exists without regularization,\!\footnote{Conditions for the existence of the inverse can be found, for instance, in Ref.~\onlinecite{Song2013} and in Section~\ref{ssec:Finite-dimensional feature space}.} can be verified as follows:
\begin{equation*}
    \innerprod{\mathcal{A}_k f}{g}_\nu = \innerprod{\cov[YX] f}{g}_{\rkhs[Y]} = \innerprod{f}{\cov[XY] g}_{\rkhs[X]} = \innerprod{f}{\ko[k] g}_{\mu}.
\end{equation*}
We have thus shown that the eigenvalue problem for the computation of coherent sets and the CCA eigenvalue problem are equivalent, provided that the RKHS is an invariant subspace of $ \mathcal{T}_k $. Although this is in general not the case---depending on the kernel the RKHS might be low-dimensional (e.g., for a polynomial kernel), but could also be infinite-dimensional and isometrically isomorphic to $ L^2 $ (e.g., for the Gaussian kernel)---, we can use the kernel-based formulation as an approximation and solve it numerically to obtain coherent sets. This is the mathematical justification for the claim that CCA detects coherent sets, which will be corroborated by numerical results in Section~\ref{sec:Numerical results}.

\subsection{Coherent mode decomposition}

Borrowing ideas from \emph{dynamic mode decomposition} (DMD) \cite{Schmid10,TRLBK14}, we now introduce a method that approximates eigenfunctions or eigenmodes of the forward-backward dynamics using linear basis functions and refer to it as \emph{coherent mode decomposition} (CMD)---a mixture of CCA and DMD.\!\footnote{In fact, the method described below is closer to TICA than DMD, but other variants can be derived in the same fashion, using different combinations of covariance and cross-covariance operators.} The relationships between DMD and TICA (including their extensions) and transfer operators are delineated in Refs.~\onlinecite{KNKWKSN18, KSM17}. DMD is often used for finding coherent structures in fluid flows, dimensionality reduction, and also prediction and control; see Ref.~\onlinecite{KBBP16} for an exhaustive analysis and potential applications.

Let us assume we have high-dimensional time-series data but only relatively few snapshots. That is, $ \mathbf{X}, \mathbf{Y} \in \R^{d \times n} $ with $ d \gg n $, where $ \mathbf{X} = [x_1, \dots, x_n] $ and $ \mathbf{Y} = [y_1, \dots, y_n] $. This is, for instance, the case for fluid dynamics applications where the, e.g., two- or three-dimensional domain is discretized using (un)structured grids. It is important to note that this analysis is now not based on Lagrangian data as before, where we tracked the positions of particles or drifters over time, but on the Eulerian frame of reference.

Using Algorithm~\ref{alg:CCA explicit} with $ \phi(x) = x $ and $ \psi(y) = y $ is infeasible here since the resulting covariance and cross-covariance matrices would be prohibitively large; thus, we apply the kernel-based counterpart. The linear kernel $ k \colon \R^d \times \R^d \to \R $ is defined by $ k(x, x^\prime) = \phi(x)^\top \phi(x^\prime) = x^\top x^\prime $ and the Gram matrices are simply given by
\begin{equation*}
    \gram[XX] = \mathbf{X}^\top \! \mathbf{X} \quad \text{and} \quad \gram[YY] = \mathbf{Y}^\top \mathbf{Y},
\end{equation*}
where $ \gram[XX], \gram[YY] \in \R^{n \times n} $.

\begin{mdframed}[backgroundcolor=boxback,hidealllines=true]
\begin{textalgorithm} \label{alg:CMD}
Coherent mode decomposition.
\begin{enumerate}[leftmargin=3ex, itemindent=0ex, itemsep=0ex, topsep=0.5ex]
\item Choose regularization $ \varepsilon $.
\item Compute Gram matrices $ \gram[XX] $ and $ \gram[YY] $.
\item Solve the eigenvalue problem \\ $ (\gram[XX] + n \varepsilon \id)^{-1} (\gram[YY] + n \varepsilon \id)^{-1} \gram[YY] \ts \gram[XX] \ts v = \rho^2 \ts v $.
\end{enumerate}
\vspace*{1ex}
\end{textalgorithm}
\end{mdframed}

The eigenfunction $ \widehat{f} $ evaluated in an arbitrary point $ x \in \R^d $ is then given by
\begin{align*}
    \widehat{f}(x) &= \Phi(x) \ts v = [k(x_1, x), \, \dots, \, k(x_n, x)] \ts v = (\mathbf{X} v)^\top x \\
        &= \xi^\top \phi(x),
\end{align*}
where we define the \emph{coherent mode} $ \xi $ corresponding to the eigenvalue $ \rho $ by $ \xi = \mathbf{X} v $. That is, $ \xi $ contains the coefficients for the basis functions $ \phi $. Analogously, we obtain
\begin{align*}
    \widehat{g}(y) &= \tfrac{1}{\rho} \Psi(y) \ts (\gram[YY] + n \varepsilon \id)^{-1} \gram[XX] v = (\mathbf{Y} w)^\top y \\
        &= \eta^\top \psi(y),
\end{align*}
where $ w = \frac{1}{\rho} (\gram[YY] + n \varepsilon \id)^{-1} \gram[XX] v $ and $ \eta = \mathbf{Y} w $.

As mentioned above, DMD (as a special case of EDMD~\cite{WKR15}) typically uses the pseudoinverse to compute matrix representations of the corresponding operators. Nonetheless, a Tikhonov-regularized variant is described in Ref.~\onlinecite{EMBK17:DMD}.

\section{Numerical results}
\label{sec:Numerical results}

As we have shown above, many dimensionality reduction techniques or methods to analyze high-dimensional data can be regarded as eigendecompositions of certain empirical RKHS operators. We now seek to illustrate how kernel CCA results in coherent sets and potential applications of the coherent mode decomposition.

\subsection{Coherent sets}

We will first apply the method to a well-known benchmark problem, namely the Bickley jet, and then to ocean data and a molecular dynamics problem.

\subsubsection{Bickley jet}

Let us consider a perturbed Bickley jet, which is an approximation of an idealized stratospheric flow \cite{Rypina07:coherent} and a typical benchmark problem for detecting coherent sets (see, e.g., Refs.~\onlinecite{HKTH16:coherent, BK17:coherent, HSD18, FJ18:coherent}). The flow is illustrated in Figure~\ref{fig:Bickley}. For a detailed description of the model and its parameters, we refer to Ref.~\onlinecite{BK17:coherent}. Here, the state space is defined to be periodic in the $x_1$-direction with period $ 20 $. In order to demonstrate the notion of \emph{coherence}, we arbitrarily color one circular set yellow and one red and observe their evolution. The yellow set is dispersed quickly by the flow; the red set, on the other hand, moves around but barely changes shape. The red set is hence called \emph{coherent}.

\begin{figure*}
    \centering
    \begin{minipage}{0.4\textwidth}
        \centering
        \subfiguretitle{(a)} \vspace*{-0.6ex}
        \includegraphics[width=\textwidth]{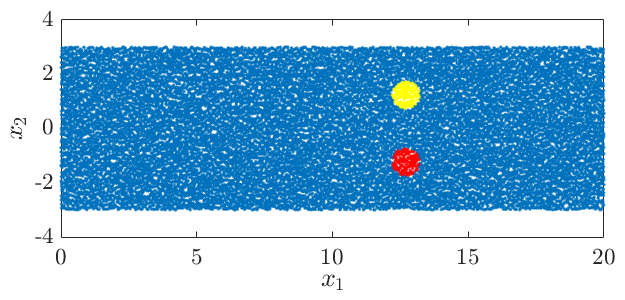}
    \end{minipage}
    \begin{minipage}{0.4\textwidth}
        \centering
        \subfiguretitle{(b)} \vspace*{-0.6ex}
        \includegraphics[width=\textwidth]{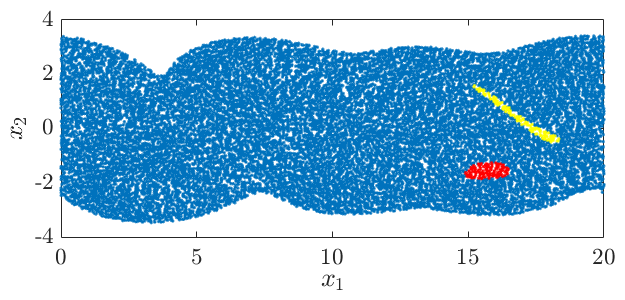}
    \end{minipage} \\
    \begin{minipage}{0.4\textwidth}
        \centering
        \subfiguretitle{(c)} \vspace*{-0.6ex}
        \includegraphics[width=\textwidth]{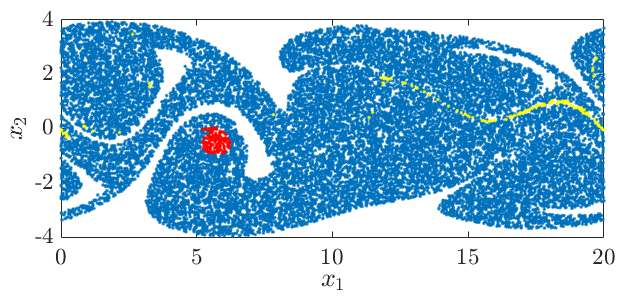}
    \end{minipage}
    \begin{minipage}{0.4\textwidth}
        \centering
        \subfiguretitle{(d)} \vspace*{-0.6ex}
        \includegraphics[width=\textwidth]{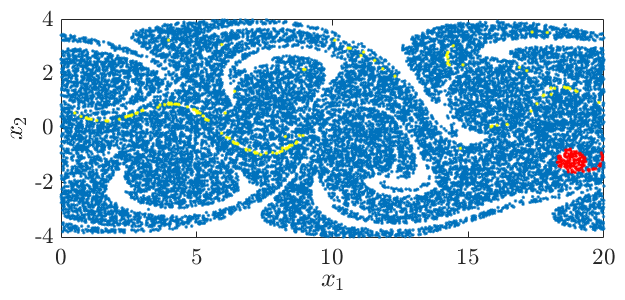}
    \end{minipage}
    \caption{Bickley jet at times (a) $t = 0$, (b) $t = 10$, (c) $t = 50$, and (d) $t = 100$ illustrating the difference between a non-coherent (yellow) and a coherent set (red). While the yellow set is dispersed after a short time, the shape of the red set remains nearly unchanged for a long time.}
    \label{fig:Bickley}
\end{figure*}

We generate 10000 uniformly distributed test points $ x_i $ in $ \inspace = [0, 20] \times [-3, 3] $ and then simulate their progression in time. For the computation of the coherent sets, we use only the start and end points of each trajectory, i.e., we define $ y_i = \Theta^\tau(x_i) $, where $ \Theta^\tau $ denotes the flow associated with the dynamical system. We set $ \tau = 40 $. From the vectors $ x_i $ and $ y_i $, we then compute the Gram matrices $ \gram[XX] $ and $ \gram[YY] $ using the same Gaussian kernel. Here, we define the bandwidth to be $ \sigma = 1 $ and the regularization parameter to be $ \varepsilon = 10^{-7} $.

A few dominant eigenfunctions are shown in Figure~\ref{fig:BickleyCS}~(a)--(d). The first eigenfunction distinguishes between the top and bottom ``half'' and the second one between the middle part and the rest. The subsequent eigenfunctions pick up combinations of the vortices. Applying $ k $-means with $ k = 9 $ to the first $ 9 $ eigenfunctions results in the coherent sets shown in Figure~\ref{fig:BickleyCS}~(e). This is consistent with the results presented in Ref.~\onlinecite{BK17:coherent} as shown in Figure~\ref{fig:BickleyCS}~(f), where we apply space-time diffusion maps to the trajectory data (comprising 40 snapshots). While the results are qualitatively the same although kernel CCA uses only two snapshots, the coherent sets computed by our approach are less noisy, which might be due to the smoothing effects of the Gaussian kernel.

Choosing a finite-dimensional feature space explicitly, as described in Section~\ref{ssec:Finite-dimensional feature space}, by selecting a set of radial basis functions whose centers are given by a regular grid leads to comparable results. Currently, only start and end points of trajectories are considered. As a result, points that drift apart and then reunite at time $ \tau $ would constitute coherent sets. Applying kernel CCA to less well-behaved systems might require more sophisticated kernels that take entire trajectories into account, e.g., by employing averaging techniques as suggested in~Ref.~\onlinecite{BK17:coherent}.

\begin{figure*}
    \centering
    \begin{minipage}{0.4\textwidth}
        \centering
        \subfiguretitle{(a) $ \rho \approx 0.98 $}
        \includegraphics[width=\textwidth]{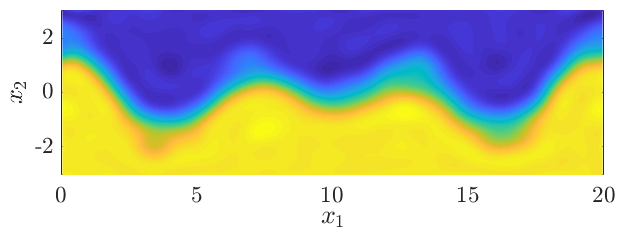} \\
        \subfiguretitle{(c) $ \rho \approx 0.78 $}
        \includegraphics[width=\textwidth]{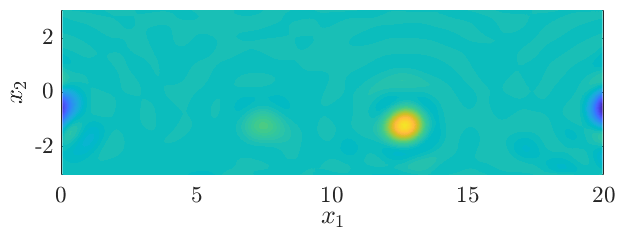}
    \end{minipage}
    \begin{minipage}{0.4\textwidth}
        \centering
        \subfiguretitle{(b) $ \rho \approx 0.87 $}
        \includegraphics[width=\textwidth]{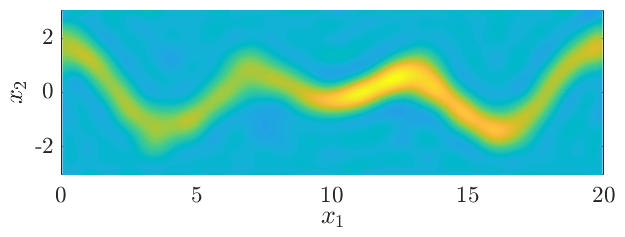} \\
        \subfiguretitle{(d) $ \rho \approx 0.75 $}
        \includegraphics[width=\textwidth]{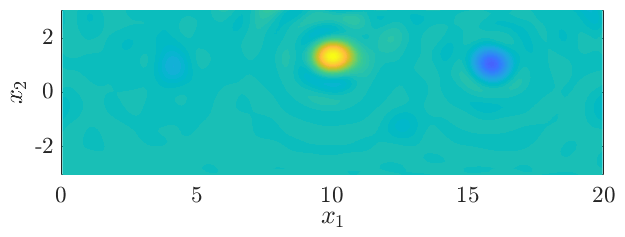}
    \end{minipage}
    \begin{minipage}{0.054\textwidth}
        \vspace*{-2.6ex}
        \includegraphics[width=\textwidth]{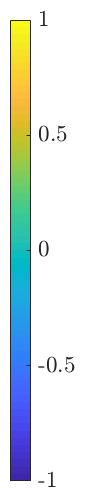}
    \end{minipage} \\
    \begin{minipage}{0.4\textwidth}
        \centering
        \subfiguretitle{(e)}
        \includegraphics[width=\textwidth]{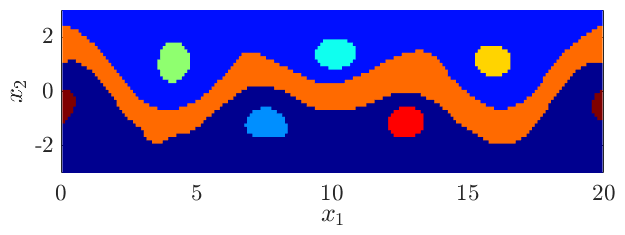}
    \end{minipage}
    \begin{minipage}{0.4\textwidth}
        \centering
        \subfiguretitle{(f)}
        \includegraphics[width=\textwidth]{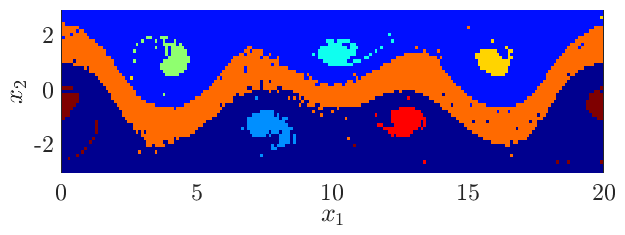}
    \end{minipage}
    \hspace*{0.054\textwidth}
    \caption{(a) First, (b) second, (c) fourth, and (d) sixth eigenfunction associated with the Bickley jet for $ \tau = 40 $. (e) $ k $-means clustering of the nine dominant eigenfunctions into nine coherent sets. The red coherent set around $ x = [12.5, -1.25]^\top $ corresponds to (but is not identical to) the red set in Figure~\ref{fig:Bickley}, where we arbitrarily selected a perfectly circular shape. (f)~Clustering obtained by applying space-time diffusion maps \cite{BK17:coherent}.}
    \label{fig:BickleyCS}
\end{figure*}

\subsubsection{Ocean data}

Ocean currents are driven by winds and tides, as well as differences in salinity. There are five major gyres as illustrated in Figure~\ref{fig:OceanData}~(a), which has been reproduced with permission of the \emph{National Ocean Service} (NOAA).\!\footnote{NOAA. What is a gyre? \url{https://oceanservice.noaa.gov/facts/gyre.html}} Our goal now is to detect these gyres from virtual buoy trajectories. In order to generate Lagrangian data, we use the \emph{OceanParcels} toolbox\footnote{OceanParcels project: \url{http://oceanparcels.org/}} (see Ref.~\onlinecite{OceanParcels17} for details) and data from the \emph{GlobCurrent} repository,\!\footnote{GlobCurrent data repository: \url{http://www.globcurrent.org/}} provided by the \emph{European Space Agency}. More precisely, our drifter computations are based on the Eulerian total current at significant wave height from the sum of geostrophic and Ekman current components, starting on the 1st of January 2016 and ending on the 31st of December 2016 with 3-hourly updates.

We place 15000 uniformly distributed virtual drifters in the oceans and let the flow evolve for one year, which thus constitutes the lag time $ \tau $. Let $ x_i $ denote the initial positions and $ y_i $ the new positions of the drifters after one year. The domain is $ \inspace = [-180^\circ, 180^\circ] \times [-80^\circ, 80^\circ] $, where the first dimension corresponds to the longitudes and the second to the latitudes. For the coherent set analysis, we select a Gaussian kernel $ k(x, x^\prime) = \exp\left(-\frac{d(x, x^\prime)^2}{2\sigma^2}\right)$ with bandwidth $ \sigma = 30 $, where $ d(x, x^\prime) $ is the distance between the points $ x $ and $ x^\prime $ in kilometers computed with the aid of the haversine formula. The regularization parameter $ \varepsilon $ is set to $ 10^{-4} $. The first two dominant eigenfunctions computed using kernel CCA are shown in Figure~\ref{fig:OceanData}~(b) and~(c) and a $k$-means clustering of the six dominant eigenfunctions in Figure~\ref{fig:OceanData}~(d). CCA correctly detects the main gyres---the splitting of the South Atlantic Gyre and the Indian Ocean Gyre might be encoded in eigenfunctions associated with smaller eigenvalues---and the Antartic Circumpolar Current. The clusters, however, depend strongly on the lag time $ \tau $. In order to illustrate the flow properties, typical trajectories are shown in Figure~\ref{fig:OceanData}~(e). The trajectories belonging to different coherent sets remain mostly separated, although weak mixing can be seen, for instance, at the borders between the red and purple and red and green clusters.

\begin{figure*}
    \centering
    \begin{minipage}{0.4\textwidth}
        \centering
        \subfiguretitle{(a)}
        \includegraphics[width=\textwidth]{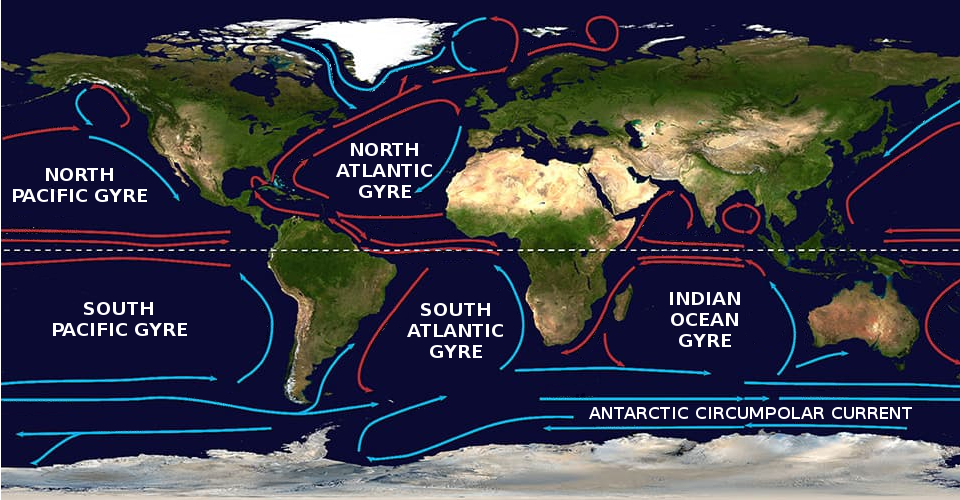}
    \end{minipage} \\[2ex]
    \begin{minipage}{0.4\textwidth}
        \centering
        \subfiguretitle{(b) $ \rho \approx 0.99 $}
        \includegraphics[width=0.95\textwidth]{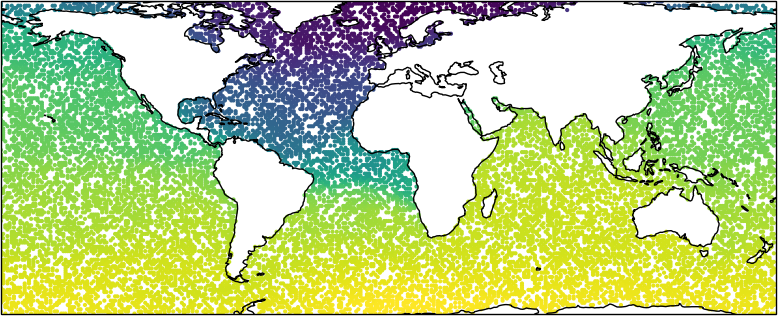}
    \end{minipage}
    \begin{minipage}{0.4\textwidth}
        \centering
        \subfiguretitle{(c) $ \rho \approx 0.98 $}
        \includegraphics[width=0.95\textwidth]{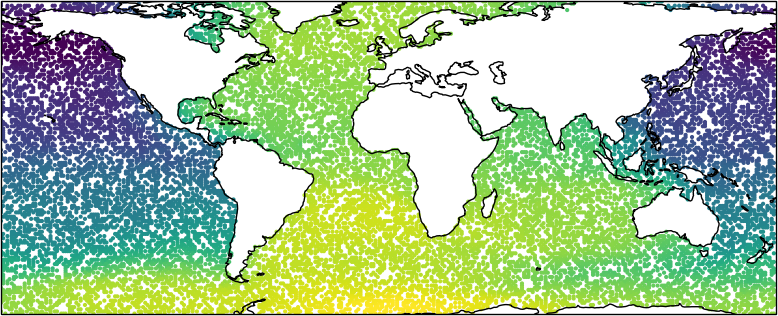}
    \end{minipage} \\[1ex]
    \begin{minipage}{0.4\textwidth}
        \centering
        \subfiguretitle{(d)}
        \includegraphics[width=0.95\textwidth]{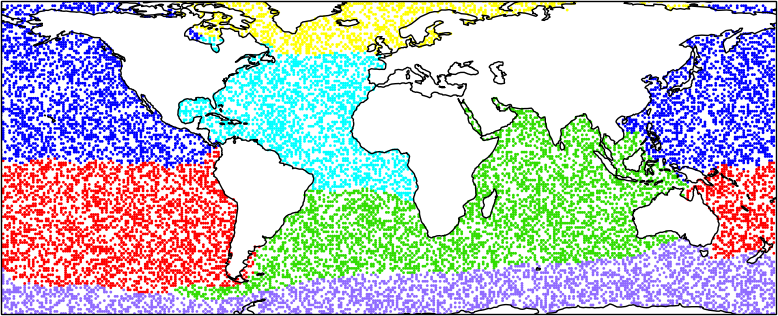}
    \end{minipage}
    \begin{minipage}{0.4\textwidth}
        \centering
        \subfiguretitle{(e)}
        \includegraphics[width=0.95\textwidth]{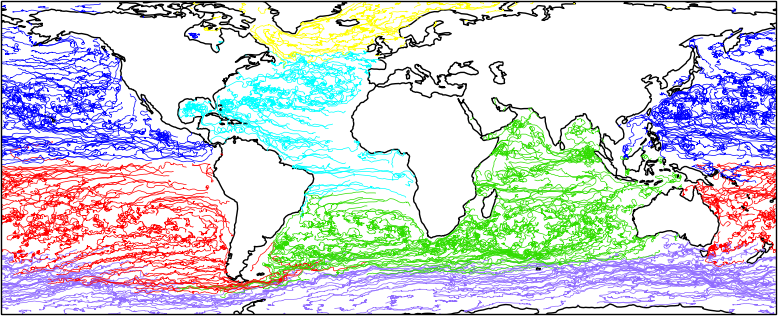}
    \end{minipage}
    \caption{(a) Illustration of the major ocean gyres (courtesy of NOAA). (b)~First and (c)~second eigenfunction. (d)~$ k $-means clustering of the first six eigenfunctions into six coherent sets. (e)~Subset of the trajectories colored according to the coherent sets.}
    \label{fig:OceanData}
\end{figure*}

\subsubsection{Time-dependent energy potential}

As a last example, we will analyze a molecular-dynamics inspired problem, namely diffusion in a time-dependent two-dimensional energy landscape, given by the stochastic differential equation
\begin{equation*}
    \dd X_t = -\nabla V(X_t, t) \ts \dd t + \sqrt{2 \ts \beta^{-1}} \ts \dd W_t,
\end{equation*}
with
\begin{align*}
    V(x, t) &= \cos\left(s \ts \arctan(x_2, x_1) - \tfrac{\pi}{2} t\right) \\
     &+ 10 \left( \sqrt{x_1^2 + x_2^2} - \tfrac{3}{2} - \tfrac{1}{2} \sin(2 \pi t) \right)^2.
\end{align*}
The parameter $ \beta $ is the dimensionless inverse (absolute) temperature, $ W_t $ a standard Wiener process, and $ s $ specifies the number of wells. This is a generalization of a potential defined in Ref.~\onlinecite{BKKBDS18}, whose wells now move periodically towards and away from the center and which furthermore slowly rotates. We set $ s = 5 $. The resulting potential for $ t = 0 $ is shown in Figure~\ref{fig:LemonSlice}~(a). Particles will typically quickly equilibrate in radial direction towards the closest well and stay in this well, which moves over time. Particles trapped in one well will remain coherent for a relatively long time. The probability of escaping and moving to another one depends on the inverse temperature: The higher $ \beta $, the less likely are transitions between~wells.

\begin{figure*}
    \centering
    \begin{minipage}{0.4\textwidth}
        \centering
        \subfiguretitle{(a)} \vspace*{-0.6ex}
        \includegraphics[width=0.9\textwidth]{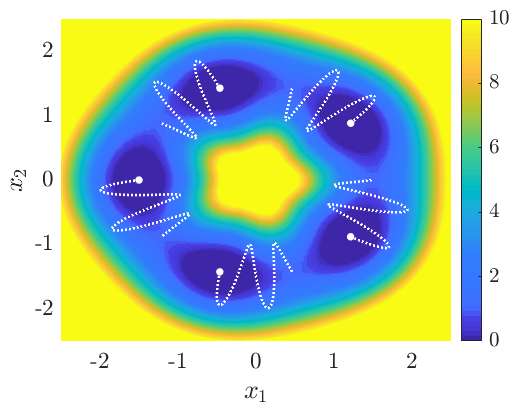}
    \end{minipage}
    \begin{minipage}{0.4\textwidth}
        \centering
        \subfiguretitle{(b)} \vspace*{-0.6ex}
        \includegraphics[width=0.9\textwidth]{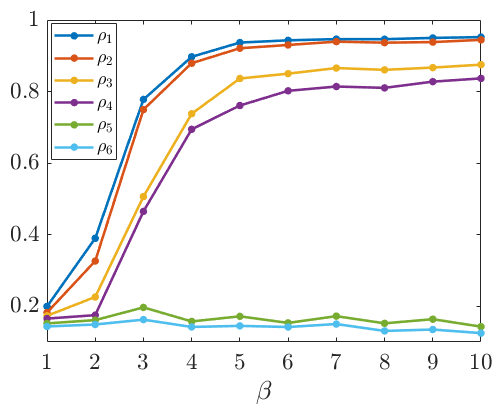}
    \end{minipage}
    \begin{minipage}{0.4\textwidth}
        \centering
        \subfiguretitle{(c)}
        \includegraphics[width=0.9\textwidth]{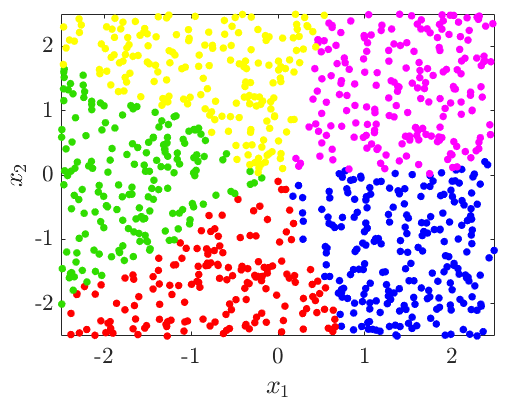}
    \end{minipage}
    \begin{minipage}{0.4\textwidth}
        \centering
        \subfiguretitle{(d)}
        \includegraphics[width=0.9\textwidth]{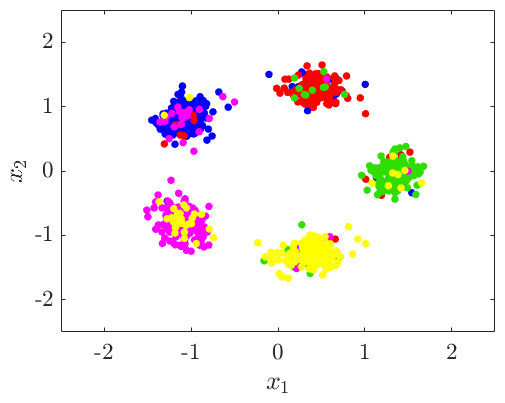}
    \end{minipage}
    \caption{(a)~Time-dependent 5-well potential for $ t = 0 $. The dotted white lines indicate the periodic movement of the centers of the wells over time. (b)~Dominant eigenvalues (averaged over multiple runs) as a functions of $ \beta $. Coherence increases with increasing inverse temperature, i.e., the eigenvalues are closer to $ 1 $ for decreasing temperature. (c)~Coherent set clustering for $ \beta = 3 $ at initial time $ t = 0 $. (d)~Corresponding clustering at $ t = 10 $. The clusters moved but are still mostly coherent save for moderate mixing.}
    \label{fig:LemonSlice}
\end{figure*}

We generate 1000 uniformly distributed test points in $ \inspace = [-2.5, 2.5] \times [-2.5, 2.5] $ and integrate the system with the aid of the Euler--Maruyama method and the step size $ h = 10^{-3} $ from $ t = 0 $ to $ t = 10 $. As before, we use only the start and end points of the trajectories and a Gaussian kernel (here, $\sigma = 1$ and $ \varepsilon = 10^{-6}$) for the coherent set analysis.

Due to the centering of the Gram matrices, the eigenvalue $ \lambda = 1 $ vanishes and---depending on the parameter $ \beta $---four eigenvalues close to one remain as illustrated in Figure~\ref{fig:LemonSlice}~(b). Figure~\ref{fig:LemonSlice}~(c) shows a clustering of the dominant four eigenfunctions for $ \beta = 3 $ based on PCCA\texttt{+} \cite{Roeblitz2013}, resulting in the expected five coherent sets. The clustering at $ t = 10 $ (see Figure~\ref{fig:LemonSlice}~(d)) illustrates that the computed sets indeed remain largely coherent.

Standard methods for the computation of metastable sets such as Ulam's method, EDMD, or their variants are in general not suitable for non-equilibrium dynamics; see also Ref.~\onlinecite{KWNS18:noneq} and Section~\ref{ssec:Relationships}.

\subsection{Coherent mode decomposition}

In order to illustrate the coherent mode decomposition outlined in Algorithm~\ref{alg:CMD}, we consider the classical von K\'arm\'an vortex street and generate data using a simple Python implementation.\!\footnote{Palabos project: \url{http://wiki.palabos.org/numerics:codes}} It is important to note that here we take into account the full trajectory data $ \{z_0, \dots, z_n\} $, where $ z_i $ is the state at time $ t = 20 \ts i $, and define $ X = [z_0, \dots, z_{n-1}] $ and $ Y = [z_1, \dots, z_n] $, whereas we generated uniformly distributed data for the coherent set analysis in the previous subsection and furthermore used only the start and end points of the trajectories. We set $ n = 100 $ and $ \varepsilon = 0.1 $. Some snapshots of the system are shown in Figure~\ref{fig:Karman}~(a)--(d). Applying CMD results in the modes depicted in Figure~\ref{fig:Karman}~(e)--(h), where the color bar is the same as in Figure~\ref{fig:BickleyCS}. As described above, we obtain two modes, denoted by $ \xi $ and $ \eta $, for each eigenvalue $ \rho $, where $ \eta $ can be interpreted as the time-lagged counterpart of $ \xi $.

\begin{figure*}
    \centering
    \begin{minipage}{1em}
        \centering
        \vspace*{2ex}
        $ \phantom{\xi} $
    \end{minipage}
    \begin{minipage}{0.24\textwidth}
        \centering
        \subfiguretitle{(a) $ t = 100 $}
        \includegraphics[width=\textwidth]{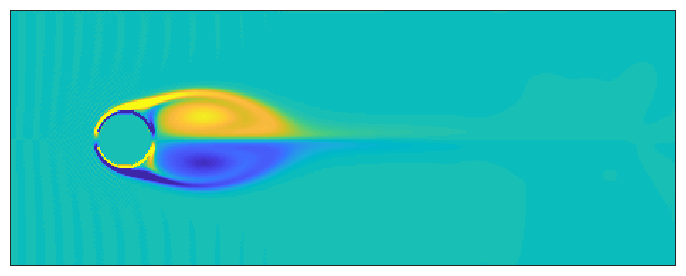}
    \end{minipage}
    \begin{minipage}{0.24\textwidth}
        \centering
        \subfiguretitle{(b) $ t = 200 $}
        \includegraphics[width=\textwidth]{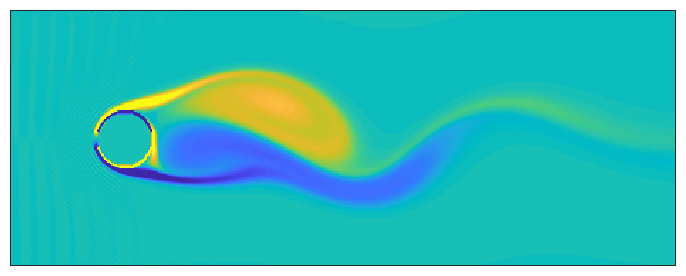}
    \end{minipage}
    \begin{minipage}{0.24\textwidth}
        \centering
        \subfiguretitle{(c) $ t = 300 $}
        \includegraphics[width=\textwidth]{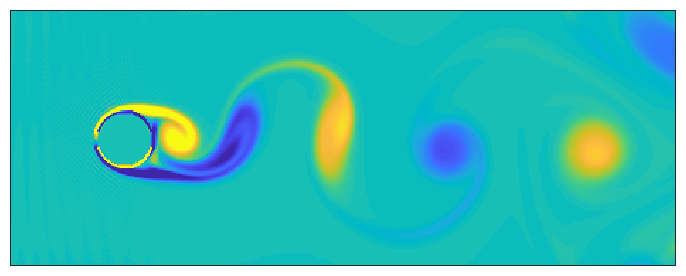}
    \end{minipage}
    \begin{minipage}{0.24\textwidth}
        \centering
        \subfiguretitle{(d) $ t = 400 $}
        \includegraphics[width=\textwidth]{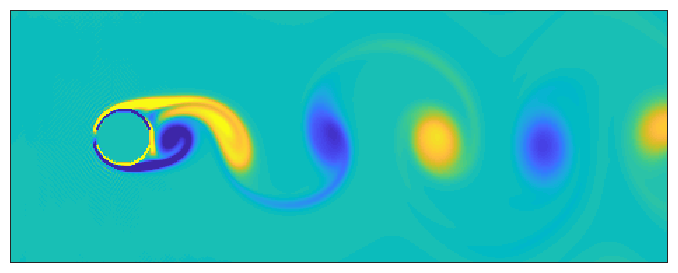}
    \end{minipage} \\[2ex]
    \begin{minipage}{1em}
        \centering
        \vspace*{2ex}
        $ \xi $ \\[13ex]
        $ \eta $
    \end{minipage}
    \begin{minipage}{0.24\textwidth}
        \centering
        \subfiguretitle{(e) $ \rho \approx 0.99  $}
        \includegraphics[width=\textwidth]{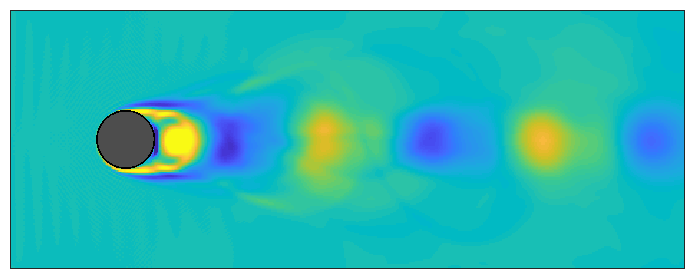} \\
        \includegraphics[width=\textwidth]{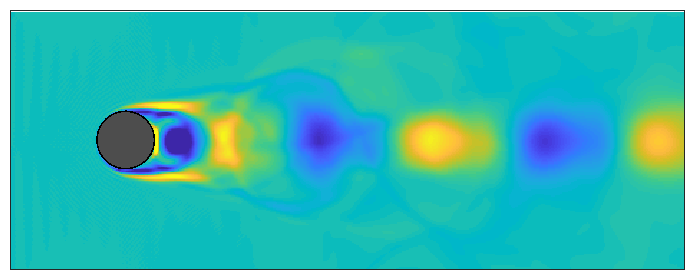}
    \end{minipage}
    \begin{minipage}{0.24\textwidth}
        \centering
        \subfiguretitle{(f) $ \rho \approx 0.98 $}
        \includegraphics[width=\textwidth]{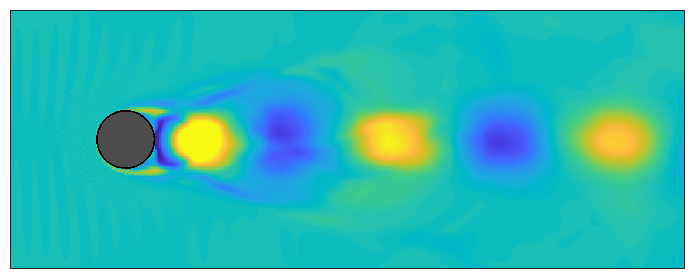} \\
        \includegraphics[width=\textwidth]{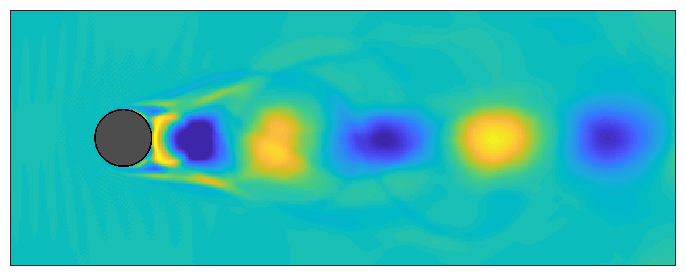}
    \end{minipage}
    \begin{minipage}{0.24\textwidth}
        \centering
        \subfiguretitle{(g) $ \rho \approx 0.94 $}
        \includegraphics[width=\textwidth]{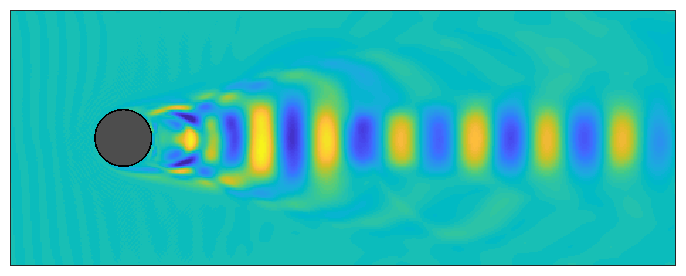} \\
        \includegraphics[width=\textwidth]{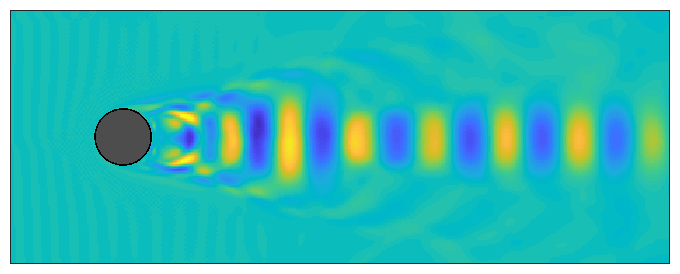}
    \end{minipage}
    \begin{minipage}{0.24\textwidth}
        \centering
        \subfiguretitle{(h) $ \rho \approx 0.92 $}
        \includegraphics[width=\textwidth]{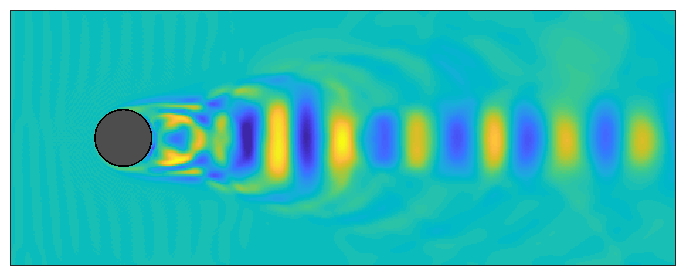} \\
        \includegraphics[width=\textwidth]{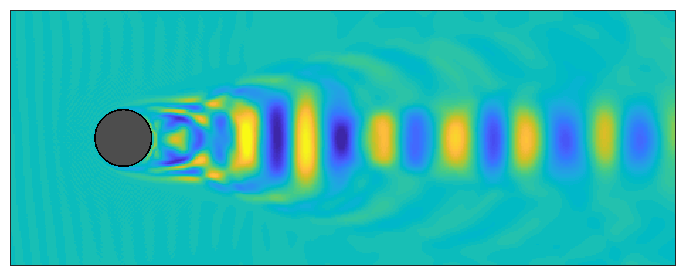}
    \end{minipage}
    \caption{(a)--(d)~Two-dimensional flow in a channel past a cylinder. Plotted are the vorticities. (e)--(h)~Three subdominant coherent modes associated with the two-dimensional flow, where the top row contains the coherent modes $ \xi $ and the bottom row the corresponding modes $ \eta $.}
    \label{fig:Karman}
\end{figure*}

For this standard DMD benchmark problem, which we chose for illustration purposes, CMD and (regularized) DMD lead to modes that look highly similar. The interpretations, however, are different. While the DMD modes, which correspond to eigenvectors, are objects that are mapped onto scalar multiples of themselves, the CMD modes, which correspond to singular vectors, encode information about how coherent structures at time $ t $ are transported by the flow to time $ t + \tau $. In fact, the DMD eigenvalues associated with the DMD modes resembling the CMD modes shown in Figure~\ref{fig:Karman} are negative and close to $ -1 $, implying periodicity. Analogously, the CMD modes $ \eta $ are akin to $ -\xi $, which also implies periodic motion. Further applications of CMD pertaining to, for instance, more complicated fluid flows or also non-sequential data, will be investigated in future work.

\section{Conclusion}
\label{sec:Conclusion}

We demonstrated that several kernel-based dimensionality reduction techniques can be interpreted as eigendecompositions of empirical estimates of certain RKHS operators. Moreover, we showed that applying CCA to Lagrangian data results in coherent sets and illustrated the efficiency of the methods using several examples ranging from fluid to molecular dynamics. This approach worked out of the box, although taking into account entire trajectories might improve the results even further, which would then necessitate dedicated kernels. In this work, we analyzed only low-dimensional benchmark problems. Nevertheless, the kernel-based algorithms can be easily applied to more complex problems and also non-vectorial domains such as graphs or strings.

As a byproduct of the coherent set analysis, we derived a method called CMD that is a hybrid of CCA and DMD (or TICA). This method can, for instance, be applied to high-dimensional fluid flow or video data. For specific problems, CMD and DMD---unsurprisingly, given the close proximity---result in highly similar modes. An interesting topic for future research would be to systematically analyze the relationships between these methods. Furthermore, as with the transfer operators and embedded transfer operators as well as their kernel-based estimates \cite{KSM17}, there are again several different combinations and variants of the proposed algorithms.

Another open problem is the influence of different regularization techniques on the numerical results. How does Tikhonov regularization compare to approaches based on pseudoinverses or other spectral filtering methods? And how do we choose the kernel and the regularization parameters in an optimal way, preferably without cross-validation? Additionally, future work includes analyzing the properties of the empirical estimate $ \widehat{S} $. Can we show convergence in the infinite-data limit? Which operators can be approximated by $ \widehat{S} $ and can we derive error bounds for the resulting eigenvalues and eigenfunctions?

We expect the results in this paper to be a starting point for further theoretical research into how RKHS operators in the context of dynamical systems could be approximated and, furthermore, how they connect to statistical learning theory. Additionally, the methods proposed here might be combined with classical modifications of CCA in order to improve the numerical performance.
The experiments here were performed using Matlab, and the methods have been partially reimplemented in Python and are available at \url{https://github.com/sklus/d3s/}.

\section*{Acknowledgements}

We would like to thank P\'eter Koltai for the Bickley jet implementation as well as helpful discussions related to coherent sets, Ingmar Schuster for pointing out similarities between CCA and kernel transfer operators, and the reviewers for many helpful suggestions for improvements. We gratefully acknowledge funding from Deutsche Forschungsgemeinschaft (DFG) through grant CRC 1114 (Scaling Cascades in Complex Systems, project ID: 235221301, projects A04 and B06), Germany's Excellence Strategy (MATH\texttt{+}: The Berlin Mathematics Research Center, EXC-2046/1, project ID: 390685689, projects AA1-2, AA1-6, and EF1-2), and European Commission through ERC CoG 772230 “ScaleCell”.

\bibliographystyle{unsrt}
\bibliography{KCA}

\end{document}